\documentclass[11pt]{amsart}
\usepackage{graphicx}
\usepackage{enumerate}
\usepackage{amsmath, amsfonts, amsthm}
\usepackage{amssymb}
\usepackage{shuffle}
\usepackage{color}
\usepackage{todonotes}
\presetkeys%
    {todonotes}%
    {inline,backgroundcolor=yellow}{}
\numberwithin{equation}{section}

\newcommand{\comm}[1]{}
\newtheorem{theorem}{Theorem}
\newtheorem{definition}[theorem]{Definition}

\newtheorem{remark}[theorem]{Remark}

\newtheorem{example}[theorem]{Example}

\numberwithin{theorem}{section}
\newtheorem*{acknowledgement}{Acknowledgement}

\newtheorem{thm}[equation]{Theorem}
\newtheorem{cor}[equation]{Corollary}
\newtheorem{lem}[equation]{Lemma}
\newtheorem{prop}[equation]{Proposition}

\theoremstyle{remark}

\newcommand{\R}{\mathbb{R}}
\newcommand{\N}{\mathbb{N}}

\newcommand{\cP}{\mathcal{P}}
\newcommand{\cI}{\mathcal{I}}

\newcommand{\cL}{\mathcal{L}}
\newcommand{\var}{\text{var}}

\newcommand{\floor}[1]{\lfloor {#1} \rfloor}

\title{Signature inversion of $C^1-$axial linear curves }
\author{Chong Liu and Shi Wang}
\date{\today}

\begin{document}

\begin{abstract}
We introduce a signature inversion scheme for $C^1$-axial linear curves which are widely used in various areas. We show that in the presence of a linear coordinate function, the derivatives of the underlying curve at any point $x$ can be recovered by tracking the signature coefficients $S_{k,l}$ with $\frac{k}{k+l} \to x$. We furthermore give a quantitative estimates for the convergence rate in this inversion scheme and establish a modulus of continuity of the signature inverse $S^{-1}$ under different topologies by using this inversion procedure.
\end{abstract}

\maketitle

\section{Introduction}

For a given $\R^d$-valued $C^1$ curve $\gamma:[a,b] \to \R^d$ with coordinate functions given by $\gamma(t) = (x_1(t), \ldots, x_d(t))$, its signature $S(\gamma)$ is defined as the following formal series in the (completed) tensor algebra $T((\R^d)) = \prod_{n=0}^\infty (\R^d)^{\otimes n}$:
$$
S(\gamma) = 1 + \sum_{n=1}^\infty S^{(n)}(\gamma), 
$$
where $S^{(n)}(\gamma) = \sum_{i_1,\ldots,i_n=1}^d S^{(n)}_{i_1,\ldots,i_n}(\gamma)e_{i_1} \otimes \cdots \otimes e_{i_n} \in (\R^d)^{\otimes n}$ with
$$
S^{(n)}_{i_1,\ldots,i_n}(\gamma) = \int_{a<t_1<\ldots<t_n<b} x^\prime_{i_1}(t_1) \ldots x^\prime_{i_n}(t_n) dt_1 \ldots dt_n
$$
being the $n$-th iterated integrals of $\gamma$.  

One of the most prominent property of the signature map $$S: C^1([a,b], \R^d) \to T((\R^d))
$$
is that it is injective up to the so-called tree-like equivalence: for two $C^1$-paths $\gamma_1$ and $\gamma_2$, they have the same signature $S(\gamma_1) = S(\gamma_2)$ if and only if they are tree-like (see \cite{Chen58}, \cite{hambly2010uniqueness} for smooth curves and \cite{BGLY16} for the general rough path case). In other words, if we use $\sim_{tree}$ to denote the tree-like equivalence relation, and let $C^1([a,b],\R^d)/\sim_{tree}$ be the corresponding quotient space of $C^1([a,b],\R^d)$ modulo $\sim_{tree}$, the signature $S: C^1([a,b],\R^d)/\sim_{tree} \to T((\R^d))$ defined on this quotient space is one-to-one. 


A simple but important example of the tree-like equivalence is the time reparameterization.
In fact, in the seminal paper \cite{BGLY16}, the authors proved a fundamental result which states that 
for certain class of curves (which are called tree-reduced curves therein), time reparameterization is the only possible tree-like relation for having the same signature. For the particular use in the present paper, we mention a special case of this result below:

\begin{definition}\label{def: axial monotone paths}
Let $\gamma:[a,b]\rightarrow \mathbb R^d$ be a $C^1$-curve with coordinate functions given by
\[\gamma(t)=\left(x_1(t),\cdots,x_d(t)\right).\]
For a given index $i\in \{1,\cdots,d\}$, we say $\gamma$ is $x_i$-monotone ($x_i$-increasing) if $x_i'(t)\neq 0$ ($x_i'(t)>0$) for any $t\in [a,b]$. We say that a $C^1$ curve is axial monotone if there is a parameterization under which the curve is $x_i$-monotone for some (but not necessarily for all) $i\in\{1, \cdots, d\}$. 
\end{definition}

\begin{thm}[a corollary of Lemma 4.6 in \cite{BGLY16}]\label{thm: axial monotone paths are tree reduced}
For two axial monotone $C^1$-curves $\gamma_1$ and $\gamma_2$ from $[a,b]$ to $\R^d$, they have the same signature $S(\gamma_1) = S(\gamma_2)$ if and only if $\gamma_1$ is a reparameterization of $\gamma_2$.
\end{thm}

An important subclass of axial monotone curves are axial linear curves: we call a $C^1$ curve $\gamma(t)=\left(x_1(t),\cdots,x_d(t)\right), t \in [a,b]$ is axial linear if there is an index $i \in \{1,\cdots,d\}$ such that its $i$-th coordinate function $x_i(t)$ is linear, and in this case we will call such $\gamma$ an $x_i$-linear curve.
Clearly,  every $x_i$-monotone curve admits a  reparameterization (induced by the inverse $x_i^{-1}$) which is $x_i$-linear. Moreover, by Theorem \ref{thm: axial monotone paths are tree reduced}, it is easy to see that for two $x_i$-linear curves $\gamma_1$ and $\gamma_2$, they have the same signature iff $\gamma_1(t) = \gamma_2(t)$ holds for all $t \in [a,b]$. In other words, if $\cL^{(i)}(\R^d)$ denotes the set of all $x_i$-linear $C^1$ curves, then the tree-like equivalence relation on $\cL^{(i)}(\R^d)$ is trivial and the signature $S: \cL^{(i)}(\R^d) \to T((\R^d))$ is one-to-one.

The above fact has a significant implication when the signature theory applies to mathematical finance and statistics, where the time series are usually sensitive for time reparameterization and therefore people need to get rid of the tree-like equivalence/reparameterization when the signature is used as a good feature map to distinguish different time series.  In fact, it is almost a common convention that people often first extend an $\R^m$-valued curve $x(t) = (x_1(t), \cdots, x_m(t))$ to an $\R^{m+1}$-valued path $\tilde x(t) = (t, x_1(t), \cdots, x_m(t))$ by adding an time component as the first coordinate, so that $\tilde x \in \cL^{(1)}(\R^{m+1})$ is an axial linear process, and then define a ``time-augmented'' signature of the original curve $x$, denoted by $\tilde S(x) = S \circ \tilde x$, as the signature of the time-augmented path $\tilde x$, so that the mapping $\tilde S$ is injective (instead of modulo the tree-like equivalence) on the $\R^m$-valued path space. We refer readers to e.g. \cite{levin2013learning}, \cite{Kalsi2020Optimal}, \cite{LyonsMcleod2022}, \cite{Dupire2022},  \cite{TD2022}, \cite{bonnier2020adapted}, \cite{Bayer2021stopping}, \cite{CuchieroGazzani2023}, \cite{Cuchiero2025}, \cite{Guo2025}, \cite{Proemel2026} for seeing how people apply the signatures of axial linear (random) curves (that is, the time-augmented curve $\tilde x$ as above) to solve various problems in mathematical finance, machine learning and statistics.

Given the importance of the signature of axial linear /monotone curves mentioned above, we have enough  motivation to study its properties. In particular, we are interested in solving the \textit{signature inversion problem} of axial linear/monotone curves. To be precise, the signature inversion problem concerns the reconstruction of a curve $\gamma$ in the quotient space $C^1([a,b],\R^d)/\sim_{tree}$ (which is  parameterized by the arc length) from its signature $S(\gamma)$; or, in other words, we aim to find a clear description of the inverse of signature $$S^{-1}: S(C^1([a,b],\R^d)) \to  C^1([a,b],\R^d)/\sim_{tree}.
$$
Such an inversion problem is widely considered as one of the most important and challenging problems in the signature theory, and lots of advanced techniques were exploited to attack this hard question. We refer readers to \cite{hambly2010uniqueness}, \cite{LyonsXu2015inversion} for a geometric perspective based on the hyperbolic/Cartan development, \cite{LyonsXu2018inversion} for the symmetrization method, \cite{XuNi2017} for a probabilistic inversion scheme and \cite{BoeGeng2015}, \cite{Geng17} for a polygonal approximation via a careful geometric partition of the target space $\R^d$.

The main purpose of this draft is to give an explicit signature inversion for any $C^1$ axial linear/monotone curve. We will show that the presence of a linear coordinate function does not only exclude the tree-like equivalence (see our above discussions), but also simplifies the signature inversion procedure significantly.  Here we state our main theorem for $\R^2$-valued axial linear curves (whose proof is given in Section \ref{sect: inversion of axial linear curves}), which already illustrates the main idea.


\begin{thm}\label{cor:main}
    Let $\gamma:[0,1]\rightarrow \mathbb R^2$ be a $C^1$ curve given by
\[\gamma(x)=\left(C x,y(x)\right),x\in [0,1],\]
where $C>0$ is a constant. Denote $e_1=(1,0), e_2=(0,1)$ the standard basis in $\mathbb R^2$, $S_{k,l} := S^{(k+l+1)}_{i_1,\ldots,i_k,2,j_{1},\ldots,j_{l}}$ with $i_1=\cdots=i_k=j_1=\cdots=j_l=1$ the coefficient of the signature $S(\gamma)$ with respect to the $(k+l+1)$-tensor $e_1^{\otimes k}\otimes e_2\otimes e_1^{\otimes l}$, and $S^{(1)}_1$ is the coefficient of $S(\gamma)$ with respect to $e_1$. Then we have
\begin{enumerate}
    \item for any $x\in [0,1]$, if $(p_n,q_n)\in \mathbb N\times \mathbb N^*$ is a rational approximation of $x$ in the sense that $p_n/q_n \to x$ and $q_n \to +\infty$ as $n \to \infty$, then $y'(x)$ can be recovered via the following explicit form
    \begin{equation}\label{eq: recovery formula}
        y'(x)=\lim_{n\rightarrow \infty}\frac{(q_n+1)!\cdot S_{p_n,q_n-p_n}}{(S^{(1)}_1)^{q_n}}.
    \end{equation}

    \item if $(p_n,q_n):[0,1]\to \mathbb N\times \mathbb N^*$ is a uniform rational approximation in the sense that $\sup_{x \in [0,1]}\bigg|\frac{p_n(x)}{q_n(x)} - x\bigg| \to 0$ and $\inf_{x \in [0,1]}q_n(x) \to +\infty$ as $n\to \infty$, then the above limit is uniform, that is,
    \[\lim_{n\to \infty}\sup_{x\in [0,1]}\Bigg|y'(x)-\frac{(q_n(x)+1)!\cdot S_{p_n(x),q_n(x)-p_n(x)}}{(S^{(1)}_1)^{q_n(x)}}\Bigg|=0.\]
\end{enumerate}
\end{thm}

Since axial monotone curves can be expressed as axial linear curves by reparameterization, the above result allows us to reconstruct axial monotone curves up to time reparameterization, which is recorded into the following main theorem.

\begin{thm}\label{thm:main}
Let $\gamma:[a,b]\rightarrow \mathbb R^d$ be a $C^1$ $x_i$-increasing curve for some index $i\in \{1,\cdots, d\}$, given by
\[\gamma(t)=\left(x_1(t),\cdots,x_d(t)\right),t\in [a,b].\]
Suppose $S_{k,l}^{(j;i)}$ denotes the coefficient of the signature $S(\gamma)$ with respect to the $(k+l+1)$-tensor $e_i^{\otimes k}\otimes e_j\otimes e_i^{\otimes l}$, and $S_i$ the coefficient on the $1$-tensor $e_i$, then for any $x \in [0,1]$ and $t_x \in [a,b]$ with $x_i(t_x) = (1-x)x_i(a) + x x_i(b)$, if  $(p_n,q_n)\in \mathbb N\times \mathbb N^*$ is a rational approximation of $x$, then for any $j \neq i$, it holds that
$$
\frac{x_j'(t_x)}{x_i'(t_x)}=\lim_{n\rightarrow \infty}\frac{(q_n+1)!\cdot S_{p_n,q_n-p_n}^{(j;i)}}{(S^{(1)}_i)^{q_n+1}}.
$$
If $(p_n,q_n):[0,1]\to \mathbb N\times \mathbb N^*$ is a uniform rational approximation, the above limit can be taken uniformly over all $x \in [0,1]$.
\end{thm}

As a byproduct, we also obtain another way of recovering the length of an axial linear/monotone curve. 
\begin{cor}\label{cor:length}
    Let $\gamma:[a,b]\rightarrow \mathbb R^d$ be a $C^1$ $x_i$-increasing curve. Then the length of $\gamma$ can be recovered by the signature via
\[L(\gamma)=\lim_{n\rightarrow \infty}\sum_{k=1}^{n}\sqrt{\sum_{j=1}^d\left(S^{(j;i)}_{k,n-k}\right)^2
}\cdot\frac{n!}{(S^{(1)}_i)^n}.\]
\end{cor}

A remarkable advantage of our signature inversion scheme is that we can easily obtain a quantitative bound on the approximation error between $x^\prime_j(x)$ and the ratio $\frac{(q_n+1)!\cdot S_{p_n,q_n-p_n}}{(S^{(1)}_i)^{q_n}}$ consisting of its signature components, provided the derivatives of $\gamma$ are H\"older continuous:

\begin{thm}\label{thm:effective bound on inversion}[see Theorem \ref{thm:quantitative}]
Let $\alpha \in (0,1]$. Assume that $\gamma \in \cL^{(i)}([0,1],\R^d)$ such that its derivative $\gamma^\prime$ is $\alpha$-H\"older continuous. Then for any $x \in [0,1]$, any $\epsilon_0 \in (0,\frac{1}{2})$, any $j \neq i$ and any rational approximation $(p_n,q_n)$ of $x$ satisfying
\begin{equation}\label{eq:approx rate of x}
    \bigg|\frac{p_n}{q_n} - x\bigg| < q_n^{-\frac{1}{2}}, \quad q_n \to +\infty
\end{equation}
as $n \to \infty$, there exists a constant $C_1$ that only depends on $\epsilon_0$, $\alpha$, $\|\gamma^\prime\|_\infty$ and $\|\gamma^\prime\|_\alpha$ (the $\alpha$-H\"older norm of $\gamma^\prime$) such that for large enough $n$,
$$
\bigg| x^\prime_j(x) - \frac{(q_n+1)!\cdot S_{p_n,q_n-p_n}}{(S^{(1)}_i)^{q_n}} \bigg|\le C_1 q_n^{(-\frac{1}{2}+\epsilon_0)\alpha}.
$$
\end{thm}

The above result reveals a quite interesting fact that the approximating rate for recovering the derivative $x_j^\prime(x)$ at $x \in [0,1]$ in our signature inversion scheme depends crucially on the convergence rate of the rational approximation $\frac{p_n}{q_n}$ to $x$, see eq. \eqref{eq:approx rate of x}. This motivates us to apply e.g. the Dirichlet/Diophantine approximation theorem from the number theory to achieve a desired convergence rate for our signature inversion scheme, see Section \ref{sec:effective} for more detailed discussions in this direction.

Another contribution of the present paper is the analysis of the regularity of the signature inversion $S^{-1}$. It is a natural question that if the signatures of two curves are close to each other, then whether these two curves also have similar traces. Of course the answer to this question depends on the topologies on the image tensor space and the path space. For instance, it is well known that $S^{-1}$ is in general not continuous when $T((\R^d))$ is endowed with the projective tensor norm $\|\cdot\|_{proj}$ and $C^1([a,b],\R^d)$ is equipped with the bounded variation norm $\|\cdot\|_{BV}$. In this case one can only expect the continuity of $S^{-1}$ on certain compact subsets. Interestingly, our signature inversion method allows to derive an explicit modulus of continuity of $S^{-1}$ on such compact subset, which has not been obtained in other literature as far as we know: 

\begin{thm}\label{thm:modulus of cont of signature inversion}[see Theorem \ref{thm:mod-cont-S-inverse}]
Let $\alpha \in (0,1]$, $\epsilon_0 \in (0,\frac{1}{2})$ and $K>0$ be given. Then for any $\epsilon >0$ and any $x_i$-linear curves $\gamma_1,\gamma_2 \in \cL^{(i)}([0,1],\R^d)$ with $\max_{i=1,2}\|\gamma_i^\prime\|_\infty \le K$ and $\max_{i=1,2}\|\gamma_i^\prime\|_\alpha \le K$, it holds that
$$
\|S(\gamma_1) - S(\gamma_2)\|_{proj} < \delta \Rightarrow \|\gamma^\prime_1 - \gamma^\prime_2\|_\infty < \epsilon,
$$
as long as $\delta = \delta(\epsilon)$ satisfies
$$
\delta < \frac{\bar C_2}{\sqrt{d}}\frac{(S^{(1)}_i)^n}{(n+1)!}\epsilon \text{ and } n > \bar C_1 \epsilon^{\frac{1}{(-\frac{1}{2}+\epsilon_0)\alpha}}
$$
for two universal constants $\bar C_1, \bar C_2$ depending only on $\alpha, \epsilon_0,\alpha$ and $K$.
\end{thm}

Besides, we also provide a novel class of ``asymptotic norms'' on $T((\R^d))$ such that the ``time-augmented'' signature $\hat S: C^1([0,1],\R^{d-1}) \to T((\R^d))$, $\hat S(\gamma) = S(\hat \gamma)$ (where $\hat \gamma(t) = (t,\gamma(t))$ is the time-augmented curve) is an isometric embedding, which offers enough regularity of the signature inversion while keeping the continuity of the signature mapping:

\begin{thm}\label{thm: isometry by as norm}[see Theorem \ref{thm:norm}]
Let $d \ge 2$ and $C^1([0,1],\R^{d-1})$ be equipped with the $C^1$-norm $\|\gamma\|_{C^1} = \|\gamma^\prime\|_\infty$ resp. the bounded variation norm $\|\gamma\|_{BV} = \|\gamma^\prime\|_{L^1}$. Then there exists an asymptotic supremum (semi)norm $\|\cdot\|_{AS}$ resp. an asymptotic $L^1$-(semi)norm $\|\cdot\|_{AL^1}$ on $T((\R^d))$ such that the ``time-augmented'' signature $$
\hat S: (C^1([0,1],\R^{d-1}),\|\cdot\|_{C^1}) \to (T((\R^d)),\|\cdot\|_{AS}), \quad \gamma(t) \mapsto S(\hat \gamma) = S(t \mapsto (t,\gamma(t))
$$
resp.
$$
\hat S: (C^1([0,1],\R^{d-1}),\|\cdot\|_{BV}) \to (T((\R^d)), \|\cdot\|_{AL^1}), \quad \gamma(t) \mapsto S(\hat \gamma) = S(t \mapsto (t,\gamma(t))
$$
is an isometric embedding.
\end{thm}

All above results suggest that the signature inversion scheme given in the present paper may have a strong potential in many applications, as it may help us to better understand functions on axial linear/monotone path space by pulling them back to the signature group through the inverse map $S^{-1}$, and the resulting functions on the signature group are usually easier to study (recall that the functions on signature are ``almost linear'' due to the universal approximation property of signature) and have already been well studied in a large amount of literature.

The present paper is organized as follows:  In section \ref{sect: inversion of axial linear curves} we introduce our signature inversion scheme for axial linear/monotone curves and provide proofs for the qualitative results Theorem \ref{cor:main}, Theorem \ref{thm:main} and Corollary \ref{cor:length}. In Section \ref{sec: regularity of signature inversion} we show that this inversion scheme on the image of certain compact subsets of axial linear curves is continuous, and then define two asymptotic norms on the tensor algebra $T((\R^d))$ such that the ``time-augmented'' signature $\hat S: C^1([0,1],\R^{d-1}) \to T((\R^d))$ is an isometry when the topology on the image set are induced by these norms, which proves Theorem \ref{thm: isometry by as norm}.  In Section \ref{sec:effective} we give quantitative estimates for the convergence rate of our signature scheme and the modulus of the signature inversion $S^{-1}$ for axial linear curves with H\"older continuous derivatives, which proves Theorem \ref{thm:effective bound on inversion} and Theorem \ref{thm:modulus of cont of signature inversion}.

\textbf{Notations:} 
In this paper, we use $\N = \{0,1,2,\cdots\}$ to denote the set of nonnegative integers and use $\N^* = \{1,2,\cdots,\}$ for the collection of positive integers. For $x \ge 0$, $\floor{x} = \sup\{n \in \N: n \le x\}$.
For two sequences of real numbers $a(n), n \in \N$ and $b(n),n\in \N$, we adopt the notations that
\[a(n)\lesssim b(n) \iff \exists n_0>0 \textrm{ such that } \forall n>n_0, a(n)\leq b(n),\]
\[a(n)\sim b(n)\iff \lim_{n\to \infty}\frac{a(n)}{b(n)}=1.\]
For $d \in \N^*$, the Euclidean space $\R^d$ is always endowed with the classical Euclidean norm $|v| = \sqrt{\sum_{i=1}^d x_i^2}$ for $v = (x_1, \cdots, x_d) \in \R^d$. The canonical basis on $\R^d$ will be denoted as $e_1, \ldots, e_d$. For $n \in \N^*$, the symbol $(\R^d)^{\otimes n}$ means the $n$-th tensor product of $\R^d$, note that $\{e_{i_1} \otimes \cdots \otimes e_{i_n}: i_1, \ldots, i_n \in \{1,\ldots,d\}\}$ forms a basis for $(\R^d)^{\otimes n}$. Every element $a$ in $(\R^d)^{\otimes n}$ will be called an $n$-tensor, and it can be uniquely expressed as $a = \sum_{i_1,\ldots,i_n=1}^d a_{i_1,\ldots,i_n}e_{i_1} \otimes \cdots \otimes e_{i_n}$ for $a_{i_1,\ldots,i_n} \in \R$ being the coefficient of $a$ with respect to $e_{i_1} \otimes \cdots \otimes e_{i_n}$.
For a given compact interval $[a,b] \subset \R$, we use $C^1([a,b],\R^d)$ to denote the space of all differentiable $\R^d$-valued curves on $[a,b]$ whose derivatives are continuous. For a $C^1$-curve $\gamma \in C^1([a,b],\R^d)$, it will be expressed as $\gamma(t) = (x_1(t), \cdots, x_d(t))$ such that $x_i(t) \in C^1([a,b],\R)$, $i=1,\ldots,d$ are the coordinate functions of $\gamma$. If $[a.b] = [0,1]$, then we simply write $C^1(\R^d)$ for $C^1([0,1],\R^d)$.
For $\gamma \in C^1([a,b],\R^d)$, we write $\overleftarrow{\gamma}(t) = \gamma (a + (b-t))$ as the time reversal of $\gamma$. For $\gamma \in C^1([a,b], \R^d)$ and $\bar \gamma \in C^1([c,d], \R^d)$, we say that $\bar \gamma$ is a reparameterization of $\gamma$ if there is an increasing homeomorphism $r:[a,b] \to [c,d]$ such that $\gamma = \bar \gamma \circ r$.
We also set $\mathcal P^1(\mathbb R^d)$ the space of all $C^1$-curves in $\R^d$ defined on $[0,1]$ up to translation equivalence, that is,
\[\mathcal P^1(\mathbb R^d)=C^1(\R^d)/\sim_{trans}\]
where two curves are equivalent under $\sim_{trans}$ iff they differ by a constant vector in $\mathbb R^d$. For this reason, we can regard any element $\gamma$ in $\mathcal P^1(\mathbb R^d)$ as a $C^1$-curve starting from the origin, i.e. satisfying $\gamma(0)=0$.
For a compact interval $[a,b]$, the symbol $L^2([a,b])$
denotes the space of all $\R$-valued square-integrable functions defined on $[a,b]$, i.e., $L^2([a,b]) = \{f:[a,b] \to \R: \|f\|_{L^2} = \sqrt{\int_a^b f^2(s) ds} < \infty \}$. For $f,g \in L^2([a,b])$, their inner product in $L^2([a,b])$ is written as $(f,g)_{L^2} = \int_a^b f(s) g(s) ds$. For a function $f:[a,b] \to \R$, we write $\|f\|_\infty = \sup_{s \in [a,b]}|f(s)|$.

\begin{acknowledgement} This work is partially supported by the National Key R\&D Program of China (No. 2023YFA1010900).
\end{acknowledgement}

\section{Inversion of axial linear curves}\label{sect: inversion of axial linear curves}

\subsection{Preliminaries on signatures}
We start by introducing the basic definitions and properties of signatures which will be used throughout the whole article. The stuff contained in this subsection is standard in the signature theory literature, we refer to \cite{Lyons1998}, \cite{LyonsQian2007}, \cite{Lyons2007}, \cite{FrizVictoir2010} for a more detailed exposition.

As we have mentioned at the beginning of the introduction, for a given $\gamma \in C^1([a,b],\R^d)$ with coordinate functions given by $\gamma(t) = (x_1(t), \ldots, x_d(t))$, its signature $S(\gamma)$ is defined as the following formal series in the (completed) tensor algebra $T((\R^d)) = \prod_{n=0}^\infty (\R^d)^{\otimes n}$:
$$
S(\gamma) = 1 + \sum_{n=1}^\infty S^{(n)}(\gamma), 
$$
where $S^{(n)}(\gamma) = \sum_{i_1,\ldots,i_n=1}^d S^{(n)}_{i_1,\ldots,i_n}(\gamma)e_{i_1} \otimes \cdots \otimes e_{i_n} \in (\R^d)^{\otimes n}$ with
$$
S^{(n)}_{i_1,\ldots,i_n}(\gamma) = \int_{a<t_1<\ldots<t_n<b} x^\prime_{i_1}(t_1) \ldots x^\prime_{i_n}(t_n) dt_1 \ldots dt_n 
$$
being the $n$-th iterated integrals of $\gamma$. We will call this real number $S^{(n)}_{i_1,\ldots,i_n}(\gamma)$ the coefficient of $S(\gamma)$ with respect to the tensor basis $e_{i_1} \otimes \cdots \otimes e_{i_n}$, which can also be denoted as $S^{(n)}_{i_1,\ldots,i_n}(\gamma) = \langle e_{i_1} \otimes \cdots \otimes e_{i_n}, S(\gamma) \rangle$, where $\langle e_{i_1} \otimes \cdots \otimes e_{i_n}, a \rangle$ denotes the $(i_1, \ldots, i_n)$-coordinate projection of $a \in T((\R^d))$. If the curve $\gamma$ is fixed, we also write $S^{(n)}_{i_1,\ldots,i_n}(\gamma)$ simply as $S^{(n)}_{i_1,\ldots,i_n}$.

For simplicity, we will assume $[a,b] = [0,1]$ unless otherwise stated. From the very definition of signature we can immediately see that the signature mapping is invariant under the translation, that is, if $\gamma \sim_{trans} \bar \gamma$, then $S(\gamma) = S(\bar \gamma)$. Therefore, we can and will assume that the signature mapping $S$ is defined on the quotient space $\cP^1(\R^d)$.

Below we list some (but very incomplete) basic properties on the signature $S(\gamma)$ for $\gamma \in \cP^1(\R^d)$ which will be needed in this section. All of them can be found in any literature on signature theory, e.g., \cite{Lyons2007}, \cite{FrizVictoir2010} and \cite{Chevyrev2016chf}.

\begin{itemize}
    \item For $\gamma \in \cP^1(\R^d)$, its signature $S(\gamma)$ actually takes value in a group $G(\R^d)$ embedded in the (completed) tensor algebra $T((\R^d))$ (which is called the character group) whose group operation is the tensor product $\otimes$ and the neutral element is $(1,0,0,\cdots) \in T((\R^d))$. Let $S(\gamma)^{-1} \in G(\R^d) \subset T((\R^d))$ be the inverse of $S(\gamma)$, then it holds that
    $$
    S(\gamma)^{-1} = S(\overleftarrow{\gamma})
    $$
    for $\overleftarrow{\gamma} (t) = \gamma (1-t)$ being the time reversal of $\gamma$.
    \item The signature mapping $S: \cP^1(\R^d) \to T((\R^d))$ is injective/one-to-one up to the tree-like equivalence (for the definition of tree-like equivalence, see \cite{BGLY16}). As we have mentioned in the introduction, the reparameterization is a special tree-like relation and for two axial monotone curves (see Definition \ref{def: axial monotone paths}) $\gamma_1$ and $\gamma_2$, by Theorem \ref{thm: axial monotone paths are tree reduced} or \cite[Lemma 4.6]{BGLY16} we have $S(\gamma_1) = S(\gamma_2)$ iff $\gamma_1$ is a reparameterization of $\gamma_2$.
\end{itemize}

In the present paper we will focus on a particular subclass of axial monotone curves, namely the axial linear curves which are defined as follows:

\begin{definition}\label{def: axial linear paths}
 For a given $i \in  \{1,\ldots, d\}$, a curve $\gamma \in \cP^1(\R^d)$ is called $x_i$-linear, if there exists a nonzero constant $C \in \R$ such that $x_i(t) = C t$ (i.e., the $i$-th coordinate function $x_i(t)$ is a linear function). The set of all $x_i$-linear curves is denoted by $\cL^{(i)}(\R^d)$. We call a curve $\gamma \in \cP^1(\R^d)$ an axial linear curve if it is $x_i$-linear for some $i=1,\ldots,d$. 
\end{definition}

Using Theorem \ref{thm: axial monotone paths are tree reduced}, it is straightforward to see that the signature mapping on $\cL^{(i)}(\R^d)$ is genuinely one-to-one, which we record as the following theorem:

\begin{thm}\label{thm: signature is injective on linear paths}
  For two $x_i$-linear curves $\gamma_1, \gamma_2 \in \cL^{(i)}(\R^d)$ for some $i=1,\ldots,d$, we have
  $$
  S(\gamma_1) = S(\gamma_2)
  $$
  if and only if $\gamma_1(t) = \gamma_2(t)$ for all $t \in [0,1]$. 
\end{thm}

Thanks to Theorem \ref{thm: signature is injective on linear paths}, the signature inversion  $S^{-1}: S(\cL^{(i)}(\R^d)) \to \cL^{(i)}(\R^d)$ is well-defined. The rest of this paper is devoted to studying this inverse $S^{-1}$ for axial linear curves.

\begin{remark}
    If $\gamma$ is $x_i$-monotone with $x_i'<0$, then we can consider the time reversal $\overleftarrow{\gamma}(t)=\gamma(1-t)$, so that $\overleftarrow{\gamma}(t)$ is $x_i$-increasing. On the other hand, we have $S(\overleftarrow{\gamma})=S(\gamma)^{-1}$, so knowing $S(\gamma)$ is equivalent to knowing $S(\overleftarrow{\gamma})$. Therefore, in the signature inversion task we may and will always assume that all $x_i$-linear/monotone curves are $x_i$-increasing. 
\end{remark}

Before we introduce our signature inversion scheme in the next subsection, we want to emphasize here that for an $x_i$-linear/monotone curve $\gamma$ we only require its $i$-th coordinate function $x_i(t)$ is linear/monotone, and there is no such restriction on other coordinate functions $x_j(t)$ for $j \neq i$. For example, the curve $\gamma(t) = (t, \sin (2\pi t)), t \in [0,1]$ is $x_1$-linear and therefore belongs to $\cL^{(1)}(\R^2)$, although its second coordinate function $t \mapsto \sin (2\pi t)$ is neither linear nor monotone on $[0,1]$.

\subsection{The inversion scheme of axial linear curves}
For simplicity, we first consider two-dimensional $x_1$-linear curves. Let $C_0>0$ be a fixed constant and for $\gamma \in \cL^{(1)}(\R^2)$ we write
\[\gamma(x)=\left(C_0x, y(x)\right), x\in[0,1],\]
where $y(x)$ is a $C^1$-function on $[0,1]$.  

For $k,l \in \N$, we will use $S_{k,l}$ to denote the coefficient $S^{(k+l+1)}_{i_1,\ldots,i_k,2,j_{1},\ldots,j_{l}}$ with $i_1=\cdots=i_k=j_1=\cdots=j_l=1$ of $S(\gamma)$ with respect to the $(k+l+1)$-tensor basis $e_1^{\otimes k}\otimes e_2\otimes e_1^{\otimes l}$, namely $S_{k,l} = \langle e_1^{\otimes k}\otimes e_2\otimes e_1^{\otimes l}, S(\gamma) \rangle$. The next proposition reveals that $S_{k,l}$ actually admits a nice expression.

\begin{prop}\label{prop:S_kl}
    For each $k,l\in \mathbb N$,  it holds that
    \[S_{k,l}=C_0^{k+l}\int_{0<s<1}\frac{s^k(1-s)^l}{k!\cdot l!}y'(s)ds.\]
\end{prop}
\begin{proof}
    The proof is a direct computation. It is easy to see that
    \begin{align*}
S_{k,l}&=C_0^{k+l}\int_{0<t_1<\cdots<t_k<s<t_{k+1}<\cdots<t_{k+l+1}<1}dt_1\cdots dt_k(y'(s)dsdt_{k+1}\cdots dt_{k+l+1}\\
        &=C_0^{k+l}\int_{0<s<1}\frac{s^k(1-s)^l}{k!\cdot l!}y'(s)ds,
    \end{align*}
    where the last equality integrates over all $t_1,\cdots, t_k, t_{k+1}, \cdots, t_{k+l+1}$ variables and it uses Fubini's theorem.
\end{proof}

In other words, the term $S_{k,l}$ is in fact the inner product between the derivative of the second coordinate function of $\gamma$ and the function $m_{k,l}(s) := \frac{s^k(1-s)^l}{k!\cdot l!}$ in $L^2([0,1])$, i.e., $S_{k,l} = (y^\prime, m_{k,l})_{L^2}$. It motivates us to take a closer look at the function $m_{k,l}(s)$. In fact, it is clear that it is related to the density of the Beta distribution $Be(k+1,l+1)$.

\begin{definition}\label{def:rho}
    For any integer pair $(k,l)\in \mathbb N\times \mathbb N$, we set $\rho_{k,l}:[0,1]\to \mathbb R$ as the function given by
    \[\rho_{k,l}(s)=  (k+l+1)!m_{k,l}(s) =    \frac{(k+l+1)!}{k!\cdot l!} s^k(1-s)^l,\]
    which is the density of the Beta distribution $Be(k+1,l+1)$.
\end{definition}

\begin{lem}\label{lem:rho}
 For any $k,l\in \mathbb N$ with $k+l>0$, the function $\rho_{k,l}$ is strictly increasing on $[0,\frac{k}{k+l}]$ and is strictly decreasing on $[\frac{k}{k+l},1]$.
\end{lem}
\begin{proof}
We take the derivative of the logarithm of $\rho_{k,l}$. We have
\begin{align*}
    \left(\log\rho_{k,l}(s)\right)'=\frac{\rho_{k,l}'(s)}{\rho_{k,l}(s)}=\frac{k}{s}-\frac{l}{1-s}=\frac{k-(k+l)s}{s(1-s)}.
\end{align*}
Thus, the lemma follows immediately.
\end{proof}

It is clear from the lemma that the function $\rho_{k,l}$ achieves its unique maximum at $s=\frac{k}{k+l}$, and one can show that its value tends to $+\infty$ as $(k+l)\rightarrow \infty$. Moreover, as $s$ gradually deviates from $\frac{k}{k+l}$, the function rapidly drops down very close to $0$. The following proposition effectively describes such a phenomenon, which turns out to be the key ingredient in the proofs of our main results.


\begin{prop}(Fast Decay)\label{prop:key}
For any $0<\epsilon_0<\frac{1}{2}$, there exists a constant $n_0>0$ such that for any $n>n_0$, and any $k,l\in \mathbb N$ with $k+l=n$, we have
    \[\rho_{k,l}(s)\leq 3n^{\frac{3}{2}}e^{-\frac{1}{18}n^{2\epsilon_0}}\]
    whenever
    \[\Bigg|s-\frac{k}{n}\Bigg|\geq n^{-\frac{1}{2}+\epsilon_0}.\]
\end{prop}
\begin{proof}
    Without loss of generality, we may assume $k\leq l$. Otherwise, we can swap $k,l$ and use the symmetry $\rho_{k,l}(s)=\rho_{l,k}(1-s)$. Next, we check the edge case that $k=0$. In this case, we have
    \[\rho_{0,n}(s)=(n+1)(1-s)^{n}.\] 
    When $s\geq n^{-\frac{1}{2}+\epsilon_0}$, according to Lemma \ref{lem:rho} we have
    \begin{align*}
        \rho_{0,n}(s)\leq (n+1)(1-n^{-\frac{1}{2}+\epsilon_0})^n&=(n+1)e^{n\ln(1-n^{-\frac{1}{2}+\epsilon_0})}\\
        &\lesssim ne^{-\frac{1}{2}n\cdot n^{-\frac{1}{2}+\epsilon_0}}\\
        &=n e^{-\frac{1}{2}n^{\frac{1}{2}+\epsilon_0}}\\
        &\lesssim 3n^{\frac{3}{2}}e^{-\frac{1}{18}n^{2\epsilon_0}},
    \end{align*}
    where the first ``asymptotic inequality $\lesssim$'' follows from the simple fact that $\ln(1-x) \le -x$ for all $x \in [0,1]$.
    Thus, the conclusion holds when $k=0$.
    
    For the remaining case, we have $1\leq k\leq \frac{n}{2}$. We may assume $s\in(0,1)$, otherwise the inequality holds trivially. We set $\alpha=\frac{1}{2}-\epsilon_0, r=\frac{k}{n}$ and $1-r=\frac{l}{n}$, then $r\in [\frac{1}{n},\frac{1}{2}]$. First, using Lemma \ref{lem:rhomax} below, we note that there exists an $n_0 \in \N$ such that for all $k,l \in \N$ with $n = k+l > n_0$ it holds that
    \[\frac{(k+l+1)!}{k!\cdot l!}\left(\frac{k}{k+l}\right)^k\left(\frac{l}{k+l}\right)^l<3n^{3/2}.\]
Consequently, we have
\begin{align}\label{eq:rho_kl}
\begin{split}
    \rho_{k,l}(s)&=\frac{(k+l+1)!}{k!\cdot l!} s^k(1-s)^l \\
    &=\frac{(k+l+1)!}{k!\cdot l!}\left(\frac{k}{k+l}\right)^k\left(\frac{l}{k+l}\right)^l\cdot \left[\left(\frac{s}{r}\right)^r\left(\frac{1-s}{1-r}\right)^{1-r}\right]^n\\
    &\lesssim 3n^{3/2}\exp{\left(nF(r,s)\right)},
\end{split}  
\end{align}
where 
\[F(r,s)=r\ln(s)+(1-r)\ln(1-s)-r\ln(r)-(1-r)\ln(1-r)\]
and $s\in (0,1), r\in(0,\frac{1}2]$.

We continue with the proof by dividing into two cases:\\
{\bf Case 1: if $r\leq \frac{1}{2}n^{-\alpha}$.} The constraint $|s-r|\geq n^{-\alpha}$ will force $s\geq r+n^{-\alpha}$. Then it follows from Lemma \ref{lem:rho} that
\[\rho_{k,l}(s)\leq \rho_{k,l}(r+n^{-\alpha})\leq \rho_{k,l}(n^{-\alpha}).\]
We now estimate the value $\rho_{k,l}(n^{-\alpha})$. Note that
\begin{align*}
    \frac{\partial F(r,s)}{\partial r}&=\ln(s)-\ln(1-s)-\ln(r)+\ln(1-r)\\
    &=\ln\left(1+\frac{s-r}{r(1-s)}\right)>0
\end{align*}
since $s>r$. So, in view of the range $r\leq \frac{n^{-\alpha}}{2}$, we obtain that
\[F(r,n^{-\alpha})\leq F\left(\frac{n^{-\alpha}}2,n^{-\alpha}\right).\]
Direct computation further shows that
\begin{align*}
    F\left(\frac{n^{-\alpha}}2,n^{-\alpha}\right)&=\frac{n^{-\alpha}}2\cdot \ln(2)+\left(1-\frac{n^{-\alpha}}2\right)\cdot \ln\left(1-\frac{n^{-\alpha}/2}{1-n^{-\alpha}/2}\right)\\
    &\lesssim \frac{\ln2-1}2\cdot n^{-\alpha}\\
    &\lesssim -\frac{1}{8}n^{-\alpha},
\end{align*}
where we used again the simple fact that $\ln(1-x) \le -x$ for all $x \in [0,1]$ to establish the first ``$\lesssim$''.
Thus, combining with \eqref{eq:rho_kl}, we have
\[\rho_{k.l}(s)\leq 3n^{3/2}\exp\left(-\frac{1}{8}n^{\frac{1}{2}+\epsilon_0}\right)\lesssim 3n^{3/2}\exp\left(-\frac{1}{18}n^{2\epsilon_0}\right)\]
in this case.\\
{\bf Case 2: if $r> \frac{1}2 n^{-\alpha}$.} The constraint $|s-r|\geq n^{-\alpha}$ now implies that $s\geq r+n^{-\alpha}$ or $s\leq r-n^{-\alpha}$, where the latter case can only occur if $r\geq n^{-\alpha}$.
Similarly, it follows from Lemma \ref{lem:rho} that
\begin{equation}\label{eq:rho_upp_case2}
    \rho_{k,l}(s)\leq \begin{cases}
        \max\{\rho_{k,l}(r+n^{-\alpha}),\rho_{k,l}(r-n^{-\alpha})\} & r\geq n^{-\alpha}\\
        \rho_{k,l}(r+n^{-\alpha}) & \frac{1}2n^{-\alpha}<r<n^{-\alpha}
    \end{cases}.
\end{equation}
We need to estimate the values $\rho_{k,l}(r\pm n^{-\alpha})$. For each fixed $r\in (\frac{n^{-\alpha}}2,\frac{1}{2}]$, we can use the Taylor expansion of $F(r,s)$ in the variable $s$ at $s=r$:
\begin{align*}
    F(r,s)=F(r,r)+\frac{\partial F}{\partial s}(r,r)(s-r)+\frac{1}{2}\cdot \frac{\partial^2 F}{\partial s^2}(r,\xi)(s-r)^2
\end{align*}
where $\xi$ is in between $r$ and $s$. Direct computation shows that
$F(r,r)=\frac{\partial F}{\partial s}(r,r)=0$, and 
\[\frac{\partial^2 F}{\partial s^2}(r,\xi)=\frac{r^2-r-(r-\xi)^2}{\xi^2(1-\xi)^2}.\]
Evaluate the above Taylor expansion at $s=r\pm n^{-\alpha}$, we have at both values (whenever they are defined) that
\begin{align}\label{eq:r_pm}
\begin{split}
    F(r,r\pm n^{-\alpha})&=\frac{1}{2}\cdot \frac{r^2-r-(r-\xi)^2}{\xi^2(1-\xi)^2}\cdot n^{-2\alpha}\\
    &\leq -\frac{1}{2}\frac{r(1-r)}{\xi^2(1-\xi)^2}\cdot n^{-2\alpha}\\
    &\leq -\frac{r}{4(r+n^{-\alpha})^{2}}\cdot n^{-2\alpha},
\end{split}  
\end{align}
where the last inequality uses $1-r\geq \frac{1}2$, $(1-\xi)^2\leq 1$, and $\xi\leq r+n^{-\alpha}$ (in both cases). By the assumption that $r>\frac{n^{-\alpha}}2$, we obtain from \eqref{eq:r_pm} that
\[F(r,r\pm n^{-\alpha})\leq -\frac{n^{-2\alpha}}{36r}\leq -\frac{n^{-2\alpha}}{18},\]
where the last inequality uses $r\leq \frac{1}{2}$. Thus, by \eqref{eq:rho_kl} and \eqref{eq:rho_upp_case2}, we have
\[\rho_{k,l}(s)\leq 3n^{\frac{3}{2}}e^{-\frac{1}{18}n^{2\epsilon_0}}\]
in this case.

Therefore, we have verified the inequality in all cases, so the proposition holds.
\end{proof}

\begin{lem}\label{lem:rhomax}
    There exists $n_0>0$ such that, if $k,l\in \mathbb N$ satisfies $n:=k+l>n_0$, then we have
    \[\frac{(k+l+1)!}{k!\cdot l!}\left(\frac{k}{k+l}\right)^k\left(\frac{l}{k+l}\right)^l<3n^{3/2}.\]
\end{lem}

\begin{proof}
    We first use the inequalities $\frac{k^k}{k!}<e^k$ and $\frac{l^l}{l!}<e^l$ and obtain
    \[\frac{(k+l+1)!}{k!\cdot l!}\left(\frac{k}{k+l}\right)^k\left(\frac{l}{k+l}\right)^l<\frac{(n+1)\cdot n!\cdot e^n}{n^n}.\]
    Then by the Stirling's formula: $n!\sim \sqrt{2\pi n}(\frac{n}{e})^n$. We have
    \[\frac{(n+1)\cdot n!\cdot e^n}{n^n}\sim \sqrt{2\pi n}(n+1)\lesssim 3n^{3/2}.\]
    Combining the two inequalities, the lemma holds.
\end{proof}

Now the key observation is that by using  Proposition \ref{prop:key} and then carefully choosing the integer pair $(k,l)$ with $n = k+l \to \infty$, we are able to control the support of the family of density functions $\rho_{k,l}$ so that the corresponding sequence of Beta distributions $Be(k+1,l+1)$ (with the density functions $\rho_{k,l}$) converges weakly to a Dirac measure $\delta_x$ for some $x \in [0,1]$, and this will allow us to recover the value of $y'(x)$ for such $x\in [0,1]$ by using the coefficients $S_{k,l}$ of the signature $S(\gamma)$. The following choices of $(k,l)$ will be sufficient for our use.

\begin{definition}\label{def: rational approx}
    Let $x\in [0,1]$. A sequence $(p_n,q_n)\in \mathbb N\times \mathbb N^*$ is called a rational approximation of $x$ if
    $\frac{p_n}{q_n}\in [0,1]$ and as $n\to +\infty$ we have
    \[\lim_{n\to +\infty}q_n\to +\infty\]
    and
    \[\Bigg|\frac{p_n}{q_n}-x\Bigg|\to 0.\]
    A sequence of maps $(p_n,q_n):[0,1]\to \mathbb N\times \mathbb N^*$ is called a uniform rational approximation if $$\frac{p_n(x)}{q_n(x)}\in [0,1]\quad\forall x\in[0,1],$$
    and as $n\to +\infty$ we have
    $$\inf_{x\in [0,1]}q_n(x)\to +\infty$$
    and
$$\sup_{x\in [0,1]}\Bigg|\frac{p_n(x)}{q_n(x)}-x\Bigg|\to 0.$$  
\end{definition}

\begin{example}\label{ex:natural-r-approx}
    For any $x\in [0,1]$, we can take $p_n(x)=\floor{nx}$ and $q_n(x)=n$. Then $(p_n,q_n)$ satisfies
    \[\Bigg|\frac{p_n(x)}{q_n(x)}-x\Bigg|< \frac{1}{q_n(x)}=\frac{1}{n},\]
    hence it is a uniform rational approximation.
\end{example}

\begin{remark}\label{rem:rational-approx}
Examples of uniform rational approximations are closely related to the Diophantine approximation, which we will discuss in further detail in Section \ref{sec:effective}. 
\end{remark}

\begin{remark}\label{rem:r-approx-extend}
    For a given $x_0\in [0,1]$. Suppose $(p_n,q_n)\in \mathbb N\times \mathbb N^*$ is a rational approximation at $x_0$, then we can always extend it to a uniform rational approximation $(p_n,q_n):[0,1]\to \mathbb N\times \mathbb N^*$, simply by setting
    \[\left(p_n(x),q_n(x)\right)=\begin{cases}
        (p_n, q_n) & x=x_0 \\
        (\floor{nx}, n-\floor{nx}) & x\neq x_0
    \end{cases}.\]
\end{remark}

The next theorem verifies that if $(p_n,q_n), n \in \N$
is a rational approximation of some $x \in [0,1]$, then the distributions $\mu_{p_n,q_n-p_n} := Be(p_n+1,q_n - p_n +1)$ converge weakly to the Dirac measure $\delta_x$.

\begin{thm} Given any continuous function $f$ on $[0,1]$. Then
\begin{enumerate}
    \item (Pointwise convergence) for any $x\in[0,1]$ and any rational approximation $(p_n,q_n)\in \mathbb N\times \mathbb N^*$ of $x$, we have
    \[\lim_{n\to \infty }\int_{0}^1f(s)d\mu_{p_n,q_n-p_n}(ds)  = \lim_{n\to \infty }\int_{0}^1\rho_{p_n,q_n-p_n}(s)f(s)ds=f(x).\]
    \item (Uniform convergence) for any uniform rational approximation $(p_n,q_n):[0,1]\to \mathbb N\times \mathbb N^*$, we have
    \[\lim_{n\to \infty }\sup_{x\in [0,1]}\left|\int_{0}^1\rho_{p_n(x),q_n(x)-p_n(x)}(s)f(s)ds-f(x)\right|=0.\]
\end{enumerate}\label{thm:convergence}
\end{thm}
\begin{proof}
We first prove (2). Since $f$ is continuous on $[0,1]$, there exists $M>0$ such that $|f(x)|\leq M$. For any $\epsilon>0$, since $f$ is uniformly continuous on $[0,1]$, there exists $\delta_0>0$ such that 
\begin{equation}\label{eq:thm-uconv-0}
    |f(x)-f(s)|<\frac{\epsilon}2
\end{equation}
whenever $x,s\in [0,1]$ satisfies $|x-s|<\delta_0$.

We denote $q_n=\inf_{x\in [0,1]}q_n(x)$, since $q_n\to +\infty$, there exists $n_1>0$ such that $q_n>n_0$ for all $n>n_1$ where $n_0$ is the constant as in Proposition \ref{prop:key}. Then for any $x\in [0,1]$, we can apply Proposition \ref{prop:key} for the parameters $\epsilon_0=1/4$, $k_n=p_n(x)$, $l_n=q_n(x)-p_n(x)$, and it follows that
\begin{equation}\label{eq:thm-uconv-1}
        \rho_{p_n(x),q_n(x)-p_n(x)}(s)\leq 3q_n(x)^{3/2}\exp\left(-\frac{1}{18}q_n(x)^{1/2}\right).
    \end{equation}
whenever $$\Bigg|s-\frac{p_n(x)}{q_n(x)}\Bigg|\geq q_n(x)^{-1/4}.$$
Since 
$$\sup_{x\in [0,1]}\Bigg|\frac{p_n(x)}{q_n(x)}-x\Bigg|\to 0,$$
there exists $n_2>0$ such that when $n>n_2$, we have
\begin{equation}\label{eq:thm-uconv-2}
    \Bigg|\frac{p_n(x)}{q_n(x)}-x\Bigg|<\frac{\delta_0}{2}
\end{equation}
for all $x\in [0,1]$.
Since $q_n\to +\infty$, there exists $n_3>0$ such that when $n>n_3$, we have 
\begin{equation}\label{eq:thm-uconv-3}
    q_n^{-1/4}<\frac{\delta_0}2.
\end{equation}
Hence when $n>\max\{n_2,n_3\}$ and $|s-x|\geq \delta_0$, we have by the triangle inequality that
 \begin{equation}\label{eq:thm-uconv-4}
       \Bigg|s-\frac{p_n(x)}{q_n(x)}\Bigg|\geq \Bigg|s-x\Bigg|-\Bigg|x-\frac{p_n(x)}{q_n(x)}\Bigg|\geq \delta_0- \frac{\delta_0}{2}\geq q_n^{-1/4}\geq q_n(x)^{-1/4},
   \end{equation} 
where the second and third inequalities use \eqref{eq:thm-uconv-2} and \eqref{eq:thm-uconv-3} respectively. Therefore, inequality \eqref{eq:thm-uconv-1} holds whenever $n>\max\{n_1,n_2,n_3\}$ and $|s-x|\geq \delta_0$.

Since the limit $3q^{3/2}\exp(-\frac{1}{18}q^{1/2})\to 0$ as $q\to +\infty$, there exists $q_0>0$ such that when $q_n>q_0$, we have
\begin{equation*}
        3q_n^{3/2}\exp\left(-\frac{1}{18}q_n^{1/2}\right)<\frac{\epsilon}{4M},
\end{equation*}
and since $q_n\to +\infty$ as $n\to +\infty$, there exists $n_4>0$ such that when $n>n_4$ we have $q_n>q_0$. So it follows from $q_n(x)\geq q_n$ that, we have
\begin{equation}\label{eq:thm-uconv-5}
        3q_n^{3/2}(x)\exp\left(-\frac{1}{18}q_n^{1/2}(x)\right)<\frac{\epsilon}{4M},
\end{equation}
for any $x\in [0,1]$ and any $n>n_4$.     

Now we can give the following estimates for any $n>\max\{n_1,n_2,n_3,n_4\}$ and any $x\in [0,1]$ that
    \begin{align*}
        &\quad\;\Bigg|\int_{0}^1\rho_{p_n(x),q_n(x)-p_n(x)}(s)f(s)ds-f(x)\Bigg|\\
        &\leq \int_{0}^1\rho_{p_n(x),q_n(x)-p_n(x)}(s)\Big|f(s)-f(x)\Big|ds\\
        &=\int_{|s-x|<\delta_0}\rho_{p_n(x),q_n(x)-p_n(x)}(s)\Big|f(s)-f(x)\Big|ds\\
        &\quad\quad+\int_{|s-x|\geq\delta_0}\rho_{p_n(x),q_n(x)-p_n(x)}(s)\Big|f(s)-f(x)\Big|ds\\
        &<\frac{\epsilon}2 \int_{|s-x|<\delta_0}\rho_{p_n(x),q_n(x)-p_n(x)}(s)ds+2M \int_{|s-x|\geq \delta_0}\rho_{p_n(x),q_n(x)-p_n(x)}(s)ds\\
        &<\frac{\epsilon}2 +2M\cdot \frac{\epsilon}{4M}\\
        &=\epsilon,
    \end{align*}
    where the second inequality uses \eqref{eq:thm-uconv-0} and the last second inequality uses \eqref{eq:thm-uconv-1}, \eqref{eq:thm-uconv-4} and \eqref{eq:thm-uconv-5}. This completes the proof of (2).

    Finally, (1) follows from (2) in view of Remark \ref{rem:r-approx-extend}.
   \end{proof} 

Now we are ready to prove Theorem \ref{cor:main}, Theorem \ref{thm:main} and Corollary \ref{cor:length}.

\subsection*{Proof of Theorem \ref{cor:main}:} 

For the curve $\gamma(x) = (C_0x, y(x)), x \in [0,1]$ in $\cL^{(1)}(\R^2)$ and for any integer pair $(k,l) \in \N \times \N$, by Proposition \ref{prop:S_kl} we have (recall that we set $m_{k,l}(s) = \frac{s^k(1-s)^l}{k!l!}$)
$$
S_{k,l} = C_0^{k+l}\int_0^1 m_{k,l}(s)y^\prime(s) ds.
$$
Invoking that $\rho_{k,l}(x) = (k+l+1)!m_{k,l}(x)$ (see Definition \ref{def:rho}), it follows that
$$
\int_{0}^1\rho_{k,l}(s)y^\prime(s)ds = \frac{(k+l+1)!}{C_0^{k+l}} S_{k,l}.
$$
From the definition of signature it is clear that $C_0 = S^{(1)}_1(\gamma)$ (for short, also denoted by $S^{(1)}_1$), so we can rewrite that
$$
\int_{0}^1\rho_{k,l}(s)y^\prime(s)ds = \frac{(k+l+1)!}{(S^{(1)}_1)^{k+l}} S_{k,l}.
$$
Now, for a given $x \in [0,1]$, suppose $(p_n,q_n) \in \N \times \N^*, n \in \N$ is a rational approximation of $x$, then for each $n$ we have
$$
\int_{0}^1\rho_{p_n,q_n-p_n}(s)y^\prime(s)ds = \frac{(q_n+1)!}{(S^{(1)}_1)^{q_n}} S_{p_n,q_n-p_n};
$$
and applying Theorem \ref{thm:convergence} to the continuous function $y^\prime(x)$, we get that
$$
\lim_{n \to \infty} \frac{(q_n+1)!}{(S^{(1)}_1)^{q_n}} S_{p_n,q_n-p_n} = \lim_{n \to \infty} \int_{0}^1\rho_{p_n,q_n-p_n}(s)y^\prime(s)ds = y^\prime(x),
$$
which gives the formula \eqref{eq: recovery formula} in Theorem \ref{cor:main}. If $(p_n(\cdot, q_n(\cdot)), n \in \N$ is a uniform rational approximation on $[0,1]$, the corresponding uniform convergence of $\frac{(q_n(\cdot)+1)!}{(S^{(1)}_1)^{q_n(\cdot)}} S_{p_n(\cdot),q_n(\cdot)-p_n(\cdot)}$ to $y^\prime(\cdot)$ can be established similarly by using Theorem \ref{thm:convergence}-(2).

\subsection*{Proof of Theorem \ref{thm:main}:}
Since every axial monotone curve can be converted to an axial linear curve after a suitable reparameterization, we can apply the above inversion scheme (for axial linear curves) to recover axial monotone curves up to time-change. Note that here the curve $\gamma$ can be defined on any compact interval $[a,b]$.

Without loss of generality we assume that the curve $\gamma \in C^1([a,b],\R^d)$ is $x_1$-increasing. Since $x_1'(t)>0$, by the inverse function theorem, there exists a $C^1$ increasing function $t(\cdot) = x_1^{-1}(\cdot): [x_1(a),x_1(b)] \to [a,b]$ inverse to $x_1(t)$. Now we reparameterize $\gamma$ by the function $t(\cdot)$, so the resulting curve is now written as
\[(\gamma\circ t)(u) =(u, \tilde x_2(u),\cdots, \tilde x_d(u)), u\in[x_1(a),x_1(b)],\]
where $\tilde x_j(u) := (x_j \circ t)(u)$ for $j \neq 1$.
Of course it holds that $S(\gamma) = S(\gamma \circ t)$ as the signature is unchanged under reparameterization. Take another reparameterization $u(\cdot): [0,1] \to [x_1(a),x_1(b)]$ by $u(s)=C_0s+x_1(a)$ with $C_0=x_1(b)-x_1(a)>0$, then we obtain a further reparameterization of $\gamma$:
\[(\gamma \circ t \circ u)(s) =\left(C_0s+x_1(a), \tilde x_2(C_0s+x_1(a)),\cdots, \tilde x_d(C_0s+x_1(a))\right),\]
for $s \in [0,1]$. 
Since the signature does not change under the translations on $\mathbb R^d$, we have 
\begin{equation}\label{eq:thm-main-1}
    S(\gamma)=S(\eta)
\end{equation}
where
\[\eta(s)=\left(C_0s, \tilde x_2(C_0s+x_1(a)),\cdots, \tilde x_d(C_0s+x_1(a))\right),\; s\in[0,1]\]
is simply the translation on the $x_1$-axis of the reparameterized curve $\gamma \circ t \circ u$ by $-x_1(a)$ unit.

For this axial linear curve $\eta \in \cL^{(1)}(\R^d)$, and $j\in \{2,\cdots, d\}$, if we denote $P_j:\mathbb R^d\to \mathbb R^2$ the natural projection onto the $(x_1,x_j)$-coordinates, then from the definition of signature, we simply have the relation that
\begin{equation}\label{eq:thm-main-2}
    S_{k,l}^{(j;1)}(\gamma)=S_{k,l}^{(j;1)}(\eta) = S_{k,l}(P_j\circ \eta),
\end{equation}
where $S_{k,l}^{(j;1)}(\gamma) =  S_{k,l}^{(j;1)}(\eta) = \langle e_1^{\otimes k}\otimes e_j\otimes e_1^{\otimes l}, S(\gamma) \rangle$ and $S_{k,l}(P_j\circ \eta) = \langle e_1^{\otimes k}\otimes e_2\otimes e_1^{\otimes l}, S(P_j \circ \eta) \rangle$.
Now the projected curve $P_i\circ\eta$ has the form
\[P_j\circ\eta(s)=\left(C_0s, \eta_j(s)\right),\]
where $\eta_j(s) = \tilde x_j( C_0 s + x_1(a))$ for $j=2,\ldots,d$ and $C_0 = x_1(b) - x_1(a) = \langle e_1, S(\gamma) \rangle = S^{(1)}_1(\gamma)$. 
Now we can apply Theorem \ref{cor:main} to obtain for any $j =2,\ldots,d$,
\begin{equation}\label{eq:thm-main-3}
   \eta_j^\prime(x) =\lim_{n\rightarrow \infty}\frac{(q_n+1)!\cdot S_{p_n,q_n-p_n}(P_j\circ\eta)}{(S^{(1)}_1(\gamma))^{q_n}} = \lim_{n\rightarrow \infty}\frac{(q_n+1)!\cdot S^{(j;1)}_{p_n,q_n-p_n}(\gamma)}{(S^{(1)}_1(\gamma))^{q_n}}
\end{equation}
whenever $(p_n,q_n)\in \mathbb N\times \mathbb N^*$ is a rational approximation of $x \in [0,1]$, and also that
\begin{equation}\label{eq:thm-main-3-u}
    \lim_{n\rightarrow \infty}\sup_{s\in [0,1]}\Bigg|\eta_j'(s)-\frac{(q_n(s)+1)!\cdot S^{(j;1)}_{p_n(s),q_n(s)-p_n(s)}(\gamma)}{(S^{(1)}_1(\gamma))^{q_n(s)}}\Bigg|=0
\end{equation}
whenever $(p_n,q_n):[0,1]\to \mathbb N\times \mathbb N^*$ is a uniform rational approximation.

Now noting that for any $j=2,\ldots,d$, 
$$
\eta_j^\prime(s) = \frac{d}{ds}\tilde x_j(C_0s + x_1(a)) = \frac{d}{ds} (x_j \circ t) (C_0s + x_1(a))
$$
with $t(\cdot) = x_1^{-1}(\cdot)$ being the inverse of $x_1(\cdot)$, by the chain rule we obtain that for $x \in [0,1]$,
$$
\eta_j^\prime(x) = x_j^\prime(t_x) \frac{C_0}{x_1^\prime(t_x)},\quad t_x := x_1^{-1}(C_0x + x_1(a)) = x_1^{-1}((1-x)x_1(a) + x x_1(b)).
$$
Finally, substituting the above into \eqref{eq:thm-main-3}, \eqref{eq:thm-main-3-u}, and using \eqref{eq:thm-main-1}, \eqref{eq:thm-main-2}, together with the fact that $C_0=x_1(b)-x_1(a)=S^{(1)}_1(\gamma)$, we obtain
\[\frac{x_j'(t_x)}{x_1'(t_x)}= \frac{1}{C_0}\eta_j^\prime(x) = \lim_{n\rightarrow \infty}\frac{(q_n+1)!\cdot S_{p_n,q_n-p_n}^{(j;1)}(\gamma)}{(S^{(1)}_1(\gamma))^{q_n+1}},\]
whenever $(p_n,q_n)\in \mathbb N\times \mathbb N^*$ is a rational approximation of $x \in [0,1]$
and note that $t_x \in [a,b]$ is the unique point such that $x_1(t_x) = (1-x)x_1(a) + xx_1(b)$.
The above limit is uniform whenever $(p_n,q_n): [0,1]\to \mathbb N\times \mathbb N^*$ is a uniform rational approximation.

Therefore, the theorem holds true for $i=1$. The general $x_i$-increasing case is similar, and the proof is now complete.
\begin{remark}
    Since $S(\gamma)$ does not depend on the parameterization, recovering the ratio $x_j'/x_i'$ is the best hope. On the other hand, we do not need the full information of $S(\gamma)$ to recover this ratio. Indeed, $\gamma'$ is uniquely determined (upto reparameterization) by the information of the signature coefficients of the tensor form $e_i$ and $e_i^{\otimes k}\otimes e_j\otimes e_i^{\otimes l}$ for all $j\neq i$ and $k,l\in \mathbb N$.
\end{remark}

\subsection*{Proof of Corollary \ref{cor:length}:}
Since both the length $L(\gamma) = \int_a^b |\gamma^\prime(s)|ds$ (which appears on the left hand side) and the signature $S(\gamma)$ (which appears on the right hand side) stay invariant under reparameterization and translation, for any $x_i$-increasing curve $\gamma$ we may assume without loss of generality that $\gamma$ is defined on $[0,1]$ and 
$$x_i(s)=C_0s,\;s\in [0,1],$$
where $C_0=x_i(1) - x_i(0) = S^{(1)}_i(\gamma) = \langle e_i, S(\gamma) \rangle >0$ by considering the reparameterization as in the proof of Theorem \ref{thm:main}. We now take the uniform rational approximation 
\[(p_n(s),q_n(s))=(\floor{ns},n)\]
as in Example \ref{ex:natural-r-approx}. According to Theorem \ref{thm:main}, we have for any $j\neq i$
\[\lim_{n\to \infty}\sup_{s\in [0,1]}\Bigg|x_j'(s)-\frac{(n+1)!\cdot S^{(j;i)}_{p_n(s),n-p_n(s)}}{C_0^{n}}\Bigg|=0.\]
By implementing $s=\frac{k}{n}$ for each $k=1,\cdots, n$, we see that for any $\epsilon>0$, there exists $n_1>0$ such that when $n>n_1$, we have
\begin{equation}\label{eq:cor-length-1}
\Bigg|\left(x_j'(k/n)\right)^2-\left(\frac{(n+1)!}{C_0^{n}}\right)^2\cdot (S^{(j;i)}_{k,n-k})^2\Bigg|<\frac{C_0\cdot \epsilon}{2d}
\end{equation}
for any $k=1,\cdots, n$, any $j\neq i$. It follows that for any $k=1,\cdots, n$ and any $n>n_1$, we have
\begin{align}\label{eq:cor-length-2}
\begin{split}
    &\quad\;\Bigg||\gamma'(k/n)|-\sqrt{C_0^2+\left(\frac{(n+1)!}{C_0^{n}}\right)^2\cdot \sum_{j\neq i}(S^{(j;i)}_{k,n-k})^2}\;\Bigg|\\
    &=\frac{\Bigg|\sum_{j=1}^d \left(x_j'(k/n)\right)^2-\left(C_0^2+\left(\frac{(n+1)!}{C_0^{n}}\right)^2\cdot \sum_{j\neq i}(S^{(j;i)}_{k,n-k})^2\right)\Bigg|}{|\gamma'(k/n)|+\sqrt{C_0^2+\left(\frac{(n+1)!}{C_0^{n}}\right)^2\cdot \sum_{j\neq i}(S^{(j;i)}_{k,n-k})^2}}\\
    &\leq \frac{1}{C_0}\sum_{j\neq i}\Bigg|\left(x_j'(k/n)\right)^2-\left(\frac{(n+1)!}{C_0^{n}}\right)^2\cdot (S^{(j;i)}_{k,n-k})^2\Bigg|\\
    &< \frac{\epsilon}2,
    \end{split}
\end{align}
where the second inequality uses the fact that the denominator is bounded below by $C_0$, and the last inequality uses \eqref{eq:cor-length-1}. 

Finally, since $|\gamma'|$ is continuous, it is Riemann integrable, and there exists $n_2>0$ such that when $n>n_2$, we have
\[\Bigg|\int_{0}^1|\gamma'(s)|ds-\sum_{k=1}^n|\gamma'(k/n)|\cdot \frac{1}{n}\Bigg|<\frac{\epsilon}{2}.\]
Combining with \eqref{eq:cor-length-2}, we have for all $n>\max\{n_1,n_2\}$ that
\[\Bigg|\int_{0}^1|\gamma'(s)|ds-\sum_{k=1}^n \sqrt{C_0^2+\left(\frac{(n+1)!}{C_0^{n}}\right)^2\cdot \sum_{j\neq i}(S^{(j;i)}_{k,n-k})^2}\cdot \frac{1}{n}\Bigg|<\epsilon.\]
Noting $S_i=C_0$ and $$L(\gamma)=\int_{0}^1|\gamma'(s)|ds,$$ we have
\begin{align*}
    L(\gamma)&=\lim_{n\rightarrow \infty}\sum_{k=1}^{n}\sqrt{C_0^2+\left(\frac{(n+1)!}{C_0^{n}}\right)^2\cdot \sum_{j\neq i}(S^{(j;i)}_{k,n-k})^2}\cdot \frac{1}{n}\\
&=\lim_{n\rightarrow \infty}\sum_{k=1}^{n}\sqrt{\sum_{j=1}^d\left(S^{(j;i)}_{k,n-k}\right)^2
}\cdot\frac{n!}{C_0^n},
\end{align*}
where the last equation uses the fact that for the linear function $x_i(s) = C_0s$, one has 
\[S^{(i;i)}_{k,n-k}= \langle e_i^{\otimes (n+1)}, S(\gamma) \rangle = \frac{C_0^{n+1}}{(n+1)!}.\]
This completes the proof by noting that $C_0 = S^{(1)}_i(\gamma)$. \\

Before moving to the next subsection to investigate the regularity of the signature inversion, let us put here two remarks regarding some similar ideas mentioned in the existing literature.

\begin{remark}\label{remark: TD}
For a two dimensional $x_1$-linear $C^1$-curve $\gamma(t) = (t, y(t))$, it can be shown that, see \cite[Lemma 2.7]{TD2022}, for the integer pair $(k,l) = (0,l)$ with $l \in \N^*$, the coefficient $S_{0,l}(\overleftarrow{\gamma}) = \langle e_2 \otimes e_1^{\otimes l}, S(\gamma)^{-1}\rangle$ satisfies that
$$
S_{0,l}(\overleftarrow{\gamma}) = (y, m_{0,l})_{L^2},
$$
for $m_{0,l}(s) = \frac{1}{l!}s^l$. Therefore, since the (shifted and normalized) Legendre polynomials $F_m(\cdot), m \in \N^*$ have the form $F_m(s) = \sum_{l \le m} a_{l,m}s^l, s \in [0,1]$ with well-known coefficients $a_{l,m}$ and these polynomials form an orthonormal basis of $L^2([0,1])$, one can express the projection of the second coordinate function $y(\cdot)$ onto $\text{span}\{F_m, m \le K\}$, namely $\sum_{m \le K}(y, F_m)_{L^2} F_m$, by using the signature coefficients $S_{0,l}(\overleftarrow{\gamma}) = (y, m_{0,l})_{L^2} = \frac{1}{l!}(y, s^l)$, for any $K \in \N^*$, and by letting $K \to \infty$ one can recover the representation of $y(\cdot) \in L^2([0,1])$ with respect to this orthonormal basis. We refer readers to \cite{TD2022} for more details on such reconstruction. So, to summarize, for an $x_i$-linear curve $\gamma = (x_1(\cdot), \cdots, x_d(\cdot))$, 
\begin{itemize}
    \item by only considering its signature coefficients $S^{(j;i)}_{k,l}$ for $k=0$ and $l \in \N^*$ one can already recover all coefficients of the $j$-th coordinate function $x_j(\cdot)$ in $L^2([0,1])$ with respect to a specified orthonormal basis (we may call it an ``$L^2$-inversion'');
    \item if we use $S^{(j;i)}_{k,l}$ for those integer pairs $(k,l) \in \N \times \N^*$ with $\frac{k}{k+l} \to x$ for some $x \in [0,1]$, then we can recover $x_j^\prime(x)$, the evaluation of $x^\prime_j(\cdot)$ at $x$, as we presented in Theorem \ref{cor:main} and Theorem \ref{thm:main}, which can be viewed as a ``pointwise inversion'';
    \item by using those signature coefficients $S^{(j;i)}_{k,n-k}$ for $k \in \N$ and $n \in \N^*$, we can recover the length of $\gamma$, see Corollary \ref{cor:length}.
\end{itemize}  
\end{remark}

\begin{remark}\label{remark: Lyons}
 In \cite{ChangLyons2019} the authors introduced a very interesting and effective insertion algorithm to solve the signature inversion problem for certain class of curves. Very roughly speaking, for an (almost everywhere) differentiable curve $\gamma: [0,1] \to \R^d$ of bounded variation which is parameterized at unit speed, for any $n \in \N^*$ and any $p = 1,\ldots, n+1$, the authors define a mapping $I_{p,n}: \R^d \to (\R^d)^{\otimes (n+1)}$ via
 $$
 I_{p,n}(v) = n!\int_{0<u_1<\ldots<u_n<1} f(u_1) \otimes \cdots \otimes f(u_{p-1}) \otimes v \otimes f(u_p) \otimes \cdots f(u_n) du_1\cdots du_n
 $$
 (where $f := \gamma^\prime$), that is, the function $I_{p,n}$ inserts a vector $v \in \R^d$ into the position between $p-1$ and $p$ in the $n$-th iterated integral of $\gamma$. Then, if $\gamma$ is linear over a sub-interval $[s,t] \subset [0,1]$, then for any $\theta \in (s,t)$, by choosing $p_n = \floor{\theta(n+2)}$ for all $n \in \N^*$ (note that $\frac{p_n}{n} \to \theta$, cf. Example \ref{ex:natural-r-approx}) and let
 $$
 v^*_{\theta,n} := \text{argmin}_{v \in \R^d, |v|=2} \|I_{p_n,n}(v) - (n+1)!S^{(n+1)}\|_{proj},
 $$
 one can recover $\gamma^\prime(\theta)$ by $\gamma^\prime(\theta) = \lim_{n \to \infty} v^*_{\theta,n}$, see \cite[Theorem 5.3]{ChangLyons2019}. In the proof of the above result, the authors used the Hoeffding's inequality for binomial distributions (see \cite[Theorem 3.1, Theorem 3.2]{ChangLyons2019}), which shares a similar idea as our approach and can be compared with Proposition \ref{prop:key} for establishing ``the deviation'' for the probabilities $\rho_{k,l}$ when $k = p_n = \floor{\theta(n+2)}$ and $l=n+1 - k$. Note that in \cite{ChangLyons2019} the insertion mapping $I_{p,n}(v)$ inserts the variable $v$ into the iterated integrals induced by a function $f$, while in our inversion scheme we use coefficients $S_{k,l}$ which can be intuitively thought of as inserting a function $f$ into the iterated integrals generated by the time variable $t$.
\end{remark}

\section{Regularity of $S^{-1}$}\label{sec: regularity of signature inversion}

In the last section we construct a signature inversion scheme for axial linear curves, which can be viewed as a mapping $\cI: G^{(i)}(\R^d) \to \cL^{(i)}(\R^d)$ with $G^{(i)}(\R^d) = S(\cL^{(i)}(\R^d))$ being the image of $x_i$-linear curves under the signature $S$, so that for any element $\mathbf S \in G^{(i)}(\R^d)$ , one has for $j \neq i$, the $j$-th coordinate function of the curve $\cI(\mathbf S) \in \cL^{(i)}(\R^d)$ satisfies that
$$
\forall  s \in [0,1], \quad \cI(\mathbf S)_j^\prime(s) = \lim_{n\rightarrow \infty}\frac{(q_n(s)+1)!\cdot S_{p_n(s),q_n(s)-p_n(s)}^{(j;i)}}{(S^{(1)}_i)^{q_n+1}}
$$
where we use the suggestive notation $S_{k,l}^{(j;i)}$ for the coefficient $\langle e_i^{\otimes k}\otimes e_j\otimes e_i^{\otimes l}, \mathbf S \rangle$ with respect to the tensor basis $e_i^{\otimes k}\otimes e_j\otimes e_i^{\otimes l}$ and $S^{(1)}_i$ for $\langle e_i, \mathbf S \rangle$, provided $(p_n(\cdot),q_n(\cdot)) \in \N \times \N^*, n \in \N$ is a (uniform) rational approximation. Of course if $\mathbf S = S(\gamma)$ for some $\gamma \in \cL^{(i)}(\R^d)$, then $\cI(\mathbf S) = \cI(S(\gamma)) = \gamma$: that is, $\cI$ gives a description of the signature inverse $S^{-1}$ on $G^{(i)}(\R^d)$.

Then there is a natural question that whether our inversion scheme $\cI$ is continuous; that is, we want to investigate that if two signatures in $G^{(i)}(\R^d)$ are closed to each other, then whether the two reconstructed curves by the inversion procedure $\cI$ also have similar traces. Of course the answer to this question depends crucially on the topologies on $G^{(i)}(\R^d)$ and $\cL^{(i)}(\R^d)$. So let us first briefly review the topologies which are used in establishing the continuity of the signature $S$ in standard literature.

\subsection{Preliminaries on topologies in signature theory}

Recall that given a finite dimensional Euclidean space $V=\mathbb R^d$ with norm $|\cdot|$, the space of $n$-tensors $V^{\otimes n}$ is naturally endowed with a projective norm $||\cdot||_{proj}$ given by
\[||v||_{proj}:=\min\left\{\sum_{j=1}^k|u^{(j)}_{1}|\cdots|u^{(j)}_{n}|\;:\;v=\sum_{j=1}^k u^{(j)}_{1}\otimes\cdots \otimes u^{(j)}_{n},u^{(j)}_{i}\in V\right\}.\]
It is well known that these projective norms are admissible (see \cite{Lyons2007}) in the sense that it satisfies
\begin{enumerate}
    \item for all $\sigma \in S(n)$, for all $v \in V^{\otimes n}$, it holds that
    $$
    \|\sigma(v)\|_{proj} = \|v\|_{proj},
    $$
    where $S(n)$ is the symmetric group on $n$ letters;
    \item for all $n,m \in \N^*$, for all $v \in V^{\otimes n}$, $w \in V^{\otimes m}$, 
    $$
    \|v \otimes w\|_{proj} = \|v\|_{proj}\|w\|_{proj}.
    $$
\end{enumerate}

Thus, we can define a ``projective norm'' on $T((V))$ by
\[||a||_{proj} :=\sum_{n=1}^\infty ||a_n||_{proj}\]
where $a=\sum_{n=0}^\infty a_n$ such that $a_n\in V^{\otimes n}$. For now, the value $||a||_{proj}$ can be $+\infty$ for a general $a \in T((V))$. However, it is known that if $a = S(\gamma)$ is the signature of some $\gamma \in \cP^1(\R^d)$, then by the well-known factorial decay of signature components (see e.g. \cite[Theorem 3.7]{Lyons2007})
\[||S^{(n)}(\gamma)||_{proj} \leq \frac{L(\gamma)^n}{n!},\]
 it holds that $\|S(\gamma)\|_{proj} < \infty$. Now we define a topology $\tau_{proj}$ on the image set $\tilde G(\R^d) := S(\cP^1(\R^d)) \subset T((\R^d))$ 
by using $\|\cdot\|_{proj}$, namely $\tau_{proj}$ is generated by the metric $\rho_{proj}(a, b) := \|a - b\|_{proj}$ for $a,b \in \tilde G(\R^d)$.

Now we turn to the path space $\cP^1(\R^d)$. On this space we consider two typical metrics, one is the uniform norm metric on the derivatives, denoted by $d_{C^1}$ such that for $\gamma_i \in \cP^1(\R^d), i =1,2$,
\[d_{C^1}(\gamma_1,\gamma_2):=\sup_{s\in[0,1]}|\gamma_1'(s)-\gamma_2'(s)|:= \|\gamma_1'(s)-\gamma_2'(s)\|_\infty.\]

The other metric on the path space $\cP^1(\R^d)$, which is in fact classical in the rough path theory, is the bounded variation metric, that is, 
$$
d_{BV}(\gamma_1, \gamma_2) = \|\gamma_1 - \gamma_2\|_{BV} := \int_0^1 |\gamma_1^\prime(t) - \gamma_2^\prime(t)| dt = \|\gamma^\prime_1 - \gamma^\prime_2\|_{L^1}
$$
for $\gamma_1, \gamma_2 \in \cP^1(\R^d)$. Clearly the topology induced by $d_{C^1}$ is finer than the counterpart induced by $d_{BV}$.

Clearly, every curve $\gamma \in \cP^1(\R^d)$ also has a finite $p$-variation for any $p \ge 1$, Recall that the $p$-variation norm of $\gamma$ is given by
$$
\|\gamma\|_{p-\var} = \sup_{\cP([0,1])} \bigg(\sum_{[s,t] \in \cP} |\gamma_t - \gamma_s|^p\bigg)^{1/p},
$$
where the supremum is taken over all possible partitions $\cP([0,1])$ of $[0,1]$. Hence we may also endow $\cP^1(\R^d)$ with the $p$-variation metric for any $p\ge 1$, i.e., $d_{p-\var}(\gamma_1,\gamma_2) = \|\gamma_1 - \gamma_2\|_{p-\var}$ and note that the above $d_{BV}$ is the $1$-variation metric.
A fundamental result in the rough path/signature theory is that the signature map is continuous for $p$-variation metric on the path space (where $p$ depends on the regularity of the underlying path) and $\tau_{proj}$ (actually, any topology induced by an admissible norm) on the tensor algebra, which is called the stability of the Lyons' extension. For our use we state a version of this theorem for $C^1$-curves below, see \cite[Theorem 3.1.3]{LyonsQian2007} or \cite[Theorem 9.10]{FrizVictoir2010}.

\begin{thm}\label{thm:cont-of-S}
    The signature map $S:\mathcal P^1(\mathbb R^d)\to (T((\mathbb R^d)),\tau_{proj})$ is locally Lipschitz continuous when $\cP^1(\R^d)$ is equipped with the $p$-variation metric $d_{p-\var}$ for any $p \in [1,2)$.
\end{thm}

However, Theorem \ref{thm:cont-of-S} implies that the signature inverse $$
S^{-1}: (G^{(i)}(\R^d), \tau_{proj}) \to (\cL^{(i)}(\R^d), d_{BV})
$$
is in general discontinuous when $\cL^{(i)}(\R^d)$ is endowed with the bounded variation metric (and therefore discontinuous for $d_{C^1}$ as well). The key ingredient here is that one can always construct a sequence of $C^1$-curves $\gamma_n, n \in \N^*$ such that they converges to some curve $\gamma$ in the $p$-variation topology for some $2 > p > 1$ but the convergence cannot happen for the bounded variation topology, then by Theorem \ref{thm:cont-of-S}, $S(\gamma_n) \to S(\gamma)$ in $\tau_{proj}$, but $\gamma_n$ does not tend to $\gamma$ due to the construction, so that $S^{-1}$ is discontinuous at $\gamma$. We give such an example below.

\begin{example}\label{example:discont of inverse}
  For every $n \in \N^*$ we define
  $$
  \gamma_n (t) = \left(t, \frac{1}{2n\pi}\cos(2n\pi t), \frac{1}{2n\pi}\sin (2n\pi t)\right)
  $$
  for $t \in [0,1]$, so that $\gamma_n \in \cL^{(1)}(\R^3)$. It is obvious that for $\gamma(t):=(t,0,0) \in \cL^{(1)}(\R^3)$, we have $\lim_{n \to \infty}\|\gamma_n - \gamma\|_\infty = 0$; moreover, for all $n \in \N^*$, we also have $\|\gamma_n\|_{BV} = \sqrt{2}$. As a consequence of the interpolation theorem \cite[Lemma 8.16]{FrizVictoir2010} it follows that
  $$
  \lim_{n \to \infty} \|\gamma_n - \gamma\|_{p-\var} = 0
  $$
  for any $p \in (1,2)$, and thus by Theorem \ref{thm:cont-of-S} it holds that $S(\gamma_n) \to S(\gamma)$. However, since $\|\gamma\|_{BV}= 1 \neq \sqrt{2} = \|\gamma_n\|_{BV}$, the sequence $\gamma_n$ cannot converge to $\gamma$ in the bounded variation metric. This means that with $\mathbf S_n := S(\gamma_n)$ and $\mathbf S  := S(\gamma)$, we have $\|\mathbf S_n - \mathbf S\|_{proj} \to 0$ while $d_{BV}(S^{-1}(\mathbf S_n), S^{-1}(\mathbf S)) = d_{BV}(\gamma_n, \gamma) \nrightarrow 0$. 
\end{example}

The above observation implies that, if $G^{(i)}(\R^d)$ is equipped with $\tau_{proj}$, then the continuity of  $S^{-1}$ can only happen when it is restricted to  certain proper subset of $G^{(i)}(\R^d)$. In the next subsection we will find such a good subset, which is related to the equicontinuity of curves.

\subsection{The regularity of $S^{-1}$ for the projective norm}
We recall the following definition.
\begin{definition}
    Let $x\in [0,1]$, we say a family of continuous maps $f_m:[0,1]\to \mathbb R^d, m\in \mathbb N^*$ is equicontinuous at some $x \in [0,1]$ if for any $\epsilon>0$, there exists $\delta>0$ (not depending on $m$) such that for any $m\in \mathbb N^*$, and any $y\in [0,1]$ with $|x-y|<\delta$, we have
    \[|f_m(y)-f_m(x)|<\epsilon.\]
    We say a family of continuous maps $f_m:[0,1]\to \mathbb R^d, m\in \mathbb N^*$ is uniformly equicontinuous if for any $\epsilon>0$, there exists $\delta>0$ (not depending on $m$) such that for any $m\in \mathbb N^*$, and any $x,y\in [0,1]$ with $|x-y|<\delta$, we have
    \[|f_m(y)-f_m(x)|<\epsilon.\]
\end{definition}

The next theorem shows that the equicontinuity is a sufficient condition to guarantee the pointwise continuity of our inversion scheme $\cI$ (and therefore, the signature inversion $S^{-1}$) with respect to the topology $\tau_{proj}$ on $G^{(i)}(\R^d)$.

\begin{thm}\label{thm:S-inv-pointwise}
    Let $\gamma_m\in \mathcal L^{(i)}(\mathbb R^d)$ be a family of $x_i$-linear $C^1$-curves in $\mathbb R^d$, where $m\in \mathbb N$. Suppose
    \begin{enumerate}
        \item the family $\gamma_m'$ ($m\in \mathbb N$) is equicontinuous at $x_0\in [0,1]$,
        \item the family $\gamma_m'$ ($m\in \mathbb N$) is uniformly bounded on $[0,1]$,
        \item $||S(\gamma_m)-S(\gamma_0)||\to 0 \textrm{ as }m\to \infty$,
    \end{enumerate}
    then $\gamma_m'(x_0)$ converges to $\gamma_0'(x_0)$ as $m\to \infty$.
\end{thm}

The idea of the proof is similar to that of Theorem \ref{thm:convergence}. Before presenting the proof, we will need the following proposition.

\begin{prop}\label{prop:tensor-ineq}
    Let $V=\mathbb R^d$ and $\{e_1,\cdots,e_d\}$ be the standard orthonormal basis on $V$. For each $k,l,n\in \mathbb N$ with $k+l=n$, each $i,j\in \{1,\cdots ,d\}$ with $i\neq j$, and any $a_{n+1}\in V^{\otimes (n+1)}$. Suppose $c^{(j;i)}_{k,l}$ is the coefficient of $a_{n+1}$ with respect to the $(n+1)$-tensor $e_i^{\otimes k}\otimes e_j\otimes e_i^{\otimes l}$, then we have
    \[|c^{(j;i)}_{k,l}|\leq ||a_{n+1}||_{proj},\]
    where $||\cdot||_{proj}$ denotes the projective norm on $V^{\otimes (n+1)}$.
\end{prop}
\begin{proof}
    Without loss of generality we can assume $||a_{n+1}||_{proj}$ is achieved by the representation $a_{n+1}=\sum_{r=1}^k u^{(r)}_{1}\otimes\cdots \otimes u^{(r)}_{n+1}$ where $u^{(r)}_{i}\in V$, then we have
    \[||a_{n+1}||_{proj}=\sum_{r=1}^k|u^{(r)}_{1}|\cdots|u^{(r)}_{n+1}|,\]
    where $|\cdot|$ denotes the Euclidean norm on $V$. Let $V^*$ be the dual vector space of $V$ and let $\varphi\in \left(V^*\right)^{\otimes {(n+1)}}$ be given by
    \[\varphi=(e_i^*)^{\otimes k}\otimes e_j^*\otimes (e_i^*)^{\otimes l},\]
    where $e_m^*$ denotes the dual vector of $e_m$.
    If follows that 
    \begin{align*}
        |c^{(j;i)}_{k,l}|&=|\varphi(a_{n+1})|\\
        &=\Bigg|\varphi( \sum_{r=1}^k u^{(r)}_{1}\otimes\cdots \otimes u^{(r)}_{n+1}) \Bigg| \\
        &=\Bigg|\sum_{r=1}^k \langle u^{(r)}_{1},e_i\rangle\cdots \langle u^{(r)}_{k},e_i\rangle\cdot  \langle u^{(r)}_{k+1},e_j\rangle\cdot \langle u^{(r)}_{k+2},e_i\rangle\cdots \langle u^{(r)}_{n+1},e_i\rangle\Bigg|\\
        &\leq \sum_{r=1}^k|u^{(r)}_{1}|\cdots|u^{(r)}_{n+1}|\\
        &=||a_{n+1}||_{proj}.
    \end{align*}
\end{proof}

Now we present the proof of Theorem \ref{thm:S-inv-pointwise}.
\begin{proof}
As the family $\gamma_m'$ is uniformly bounded, there exists $M>0$ such that
\begin{equation}\label{eq:thm-equi-1}
    |\gamma_m'(s)|\leq M
\end{equation}
holds for all $s\in [0,1]$ and all $m\in \mathbb N$.

Denote $f_m=\gamma_m'-\gamma_0'$ the continuous function on $[0,1]$, since $\gamma_m'$ is equicontinuous at $x_0\in [0,1]$, so is $f_m$. Then for any $\epsilon>0$, there exists $\delta_0>0$ such that 
\begin{equation}\label{eq:thm-equi-0}
    |f_m(s)-f_m(x_0)|<\epsilon
\end{equation}
whenever $|s-x_0|<\delta_0$ and $m\in \mathbb N^*$. 

For the given $x_0\in [0,1]$, we choose a good enough rational approximation $p_0/q_0$ of $x_0$ such that
\begin{enumerate}
    \item $p_0\in \mathbb N$, $q_0\in \mathbb N^*$,
    \item $\big|\frac{p_0}{q_0}-x_0\big|<\frac{\delta_0}2$,
    \item $q_0$ is large enough such that 
    \begin{equation}\label{eq:thm-equi-2}
       q_0^{-1/4}<\frac{\delta_0}2, \textrm{ and }\; 3q_0^{3/2}\exp\left(-\frac{1}{18}q_0^{1/2}\right)<\frac{\epsilon}{4M}.
    \end{equation}
\end{enumerate}
Then we apply Proposition \ref{prop:key} for the parameters $\epsilon_0=1/4$, $n=q_0$, $k=p_0$, $l=q_0-p_0$. Note that we may assume $q_0$ is large enough so that $q_0>n_0$ in order for us to apply Proposition \ref{prop:key} properly. We have
    \begin{equation}\label{eq:thm-equi-3}
        \rho_{p_0,q_0-p_0}(s)\leq 3q_0^{3/2}\exp\left(-\frac{1}{18}q_0^{1/2}\right)<\frac{\epsilon}{4M}
    \end{equation}
whenever $|s-\frac{p_0}{q_0}|\geq q_0^{-1/4}$, where the last inequality uses \eqref{eq:thm-equi-2}. If follows that when $s\in [0,1]$ satisfies $|s-x_0|\geq \delta_0$, we have by the triangle inequality that
 \begin{equation}\label{eq:thm-equi-4}
       \Bigg|s-\frac{p_0}{q_0}\Bigg|\geq \Bigg|s-x_0\Bigg|-\Bigg|x_0-\frac{p_0}{q_0}\Bigg|\geq \delta_0- \frac{\delta_0}{2}\geq q_0^{-1/4},
   \end{equation} 
 where the last inequality uses \eqref{eq:thm-equi-2}, thus inequality \eqref{eq:thm-equi-3} holds.
  
 Now we can give the following estimates.
    \begin{align*}
        &\quad\;\Bigg|\int_{0}^1\rho_{p_0,q_0-p_0}(s)f_m(s)ds-f_m(x_0)\Bigg|\\
        &\leq \int_{0}^1\rho_{p_0,q_0-p_0}(s)\Big|f_m(s)-f_m(x_0)\Big|ds\\
        &=\int_{|s-x|<\delta_0}\rho_{p_0,q_0-p_0}(s)\Big|f_m(s)-f_m(x_0)\Big|ds\\
        &\quad\quad+\int_{|s-x|\geq\delta_0}\rho_{p_0,q_0-p_0}(s)\Big|f_m(s)-f_m(x_0)\Big|ds\\
        &<\epsilon \int_{|s-x|<\delta_0}\rho_{p_0,q_0-p_0}(s)ds+ 4M \int_{|s-x|\geq \delta_0}\rho_{k_n,l_n}(s)ds\\
        &<\epsilon +4M\cdot \frac{\epsilon}{4M}\\
        &=2\epsilon,
    \end{align*}
    where the last third inequality uses \eqref{eq:thm-equi-0} and \eqref{eq:thm-equi-1}, the last second inequality uses \eqref{eq:thm-equi-3} and \eqref{eq:thm-equi-4}. Thus, using Proposition \ref{prop:S_kl}, we obtain
    \begin{equation}\label{eq:thm-equiv-5}
        |(q_0+1)!\left(S_{p_0,q_0-p_0}(\gamma_m)-S_{p_0,q_0-p_0}(\gamma_0)\right)-f_m(x_0)|<2\epsilon.
    \end{equation}
    Moreover, since $||S(\gamma_m)-S(\gamma_0)||\to 0$, there exists $m_0>0$ such that $\forall m>m_0$, we have
    \begin{equation}\label{eq:thm-equiv-6}
        ||S(\gamma_m)-S(\gamma_0)||_{proj}<\frac{\epsilon}{(q_0+1)!}.
    \end{equation}
    Finally, by the triangle inequality, we have
    \begin{align*}
        |f_m(x_0)|&\leq |(q_0+1)!\left(S_{p_0,q_0-p_0}(\gamma_m)-S_{p_0,q_0-p_0}(\gamma_0)\right)-f_m(x_0)|\\
        &\quad\quad+(q_0+1)!|S_{p_0,q_0-p_0}(\gamma_m)-S_{p_0,q_0-p_0}(\gamma_0)|\\
        &<2\epsilon+(q_0+1)!||S_{q_0+1}(\gamma_m)-S_{q_0+1}(\gamma_0)||_{proj}\\
        &<2\epsilon+(q_0+1)!||S(\gamma_m)-S(\gamma_0)||_{proj}\\
        &<3\epsilon.
    \end{align*}
    where the second inequality uses \eqref{eq:thm-equiv-5} and Proposition \ref{prop:tensor-ineq}, and the last inequality uses \eqref{eq:thm-equiv-6}. This shows that $f_m(x_0)\to 0$ as $m\to \infty$, which completes the proof.
\end{proof}

\begin{cor}\label{cor:S-inv-uniform}
    Let $\gamma_m\in \mathcal L^{(i)}(\mathbb R^d)$ be a family of $x_i$-linear $C^1$-curves in $\mathbb R^d$, where $m\in \mathbb N$. Suppose
    \begin{enumerate}
        \item the family $\gamma_m'$ ($m\in \mathbb N$) is uniformly equicontinuous on $[0,1]$,
        \item the family $\gamma_m'$ ($m\in \mathbb N$) is uniformly bounded on $[0,1]$,
        \item $||S(\gamma_m)-S(\gamma_0)||_{proj}\to 0 \textrm{ as }m\to \infty$,
    \end{enumerate}
    then $\gamma_m'$ uniformly converges to $\gamma_0'$ as $m\to \infty$.
\end{cor}
\begin{proof}
    By Theorem \ref{thm:S-inv-pointwise}, we see that $\gamma_m$ pointwise converges to $\gamma_0$. Since $\gamma_m$ is uniformly equicontinuous and uniformly bounded, the Arzela–Ascoli theorem assures that $\gamma_m$  uniformly converges to $\gamma_0$.
\end{proof}
\begin{remark}\label{remark:homeo}
    Corollary \ref{cor:S-inv-uniform} also directly follows from the Arzela–Ascoli theorem: By the Arzela–Ascoli theorem, any set of uniformly bounded and uniformly equicontinuous functions forms a relatively compact set, which we denote by $A\subset \mathcal L^{(i)}(\mathbb R^d)$. Let $\overline A$ be the closure of $A$ in $\mathcal L^{(i)}(\mathbb R^d)$ with respect to the $d_{C^1}$-metric.
    Then the signature map $S:\overline A\to G^{(i)}\subset T((V))$ is an injective continuous map from a compact topological space to a Hausdorff space, which then must be a homeomorphism onto its image. In particular, its inverse $S^{-1}$ when restricted to the image $S(\overline A)$ is continuous. On the other hand, this argument cannot give the modulus of continuity of $S^{-1}$ on $S(\overline A)$, which can be obtained by using our signature inversion $\cI$, see the next section.
\end{remark}

\subsection{The asymptotic norms}
In the last subsection we have shown that the topology $\tau_{proj}$ on $T((\R^d))$ induced by the projective norm $\|\cdot\|_{proj}$ is in general not enough for providing continuity of the signature inverse $S^{-1}$, unless $S^{-1}$ is restricted to the image of a equicontinuous family of curves. Therefore, we need to find another suitable norm $\|\cdot\|$ on $T((\R^d))$ such that $S: (\cL^{(i)}(\R^d), d_{C^1}) \to (G^{(i)}(\R^d),\|\cdot\|)$ or $S: (\cL^{(i)}(\R^d), d_{BV}) \to (G^{(i)}(\R^d),\|\cdot\|)$ is a homeomorphism. This will be the main content of this subsection.

We first note that for any $d \in \N^*$ with $d \ge 2$, the space $\mathcal L^{(i)}(\mathbb R^d)$ naturally splits as a product
\[\Phi:\mathcal L^{(i)}(\mathbb R^d)\cong \mathcal P^1(\mathbb R^{d-1})\times (\mathbb R \setminus \{0\}),\]
where the correspondence is explicitly given by
\[\gamma=(x_1(\cdot),\cdots,x_d(\cdot)) \mapsto \left((x_1(\cdot),\cdots,\widehat{x_i(\cdot)},\cdots, x_d(\cdot)), x_i(1)\right),\]
where $\widehat{x_i(\cdot)}$ means deleting the $i$-th coordinate function $x_i(\cdot)$. Based on this observation, from now on we will express any curve $\gamma \in \cP^1(\R^{d-1})$ by the suggestive notation $\gamma = (x_1(\cdot),\cdots,\widehat{x_i(\cdot)},\cdots, x_d(\cdot))$ when we want to embed $\gamma$ into $\cP^1(\R^d)$ by adding an additional time component into the $i$-th position, namely we define the following inclusion map
$\iota^{(i)}:\mathcal P^1(\mathbb R^{d-1})\to \mathcal L^{(i)}(\mathbb R^d)$
given by
\[\iota^{(i)}\left((x_1(\cdot), \cdots,\widehat{x_i(\cdot)},\cdots, x_d(\cdot))\right)=\left(x_1(\cdot), \cdots,x_i(\cdot),\cdots, x_d(\cdot)\right), x_i(s)=s.\]
Obviously it implies that the factor $\mathcal P^1(\mathbb R^{d-1})$ can be identified as an affine subspace of $\mathcal L^{(i)}(\mathbb R^d)$ by the above inclusion. 
We denote the image of $\iota^{(i)}$ by
\[\mathcal L^{(i)}_0(\mathbb R^d)=\{\gamma=(x_1(\cdot),\cdots,x_d(\cdot))\in \mathcal L^{(i)}(\mathbb R^d)\;|\;x_i(s)=s\}.\]
As we have mentioned in the introduction, usually people choose $i=1$ and use $\iota^{(1)}$ to get the time-augmented curve.

For each $k,l\in \mathbb N$ and $i\in\{1,\cdots,d\}$, we denote $L_{k,l}^{(i)}:T((\mathbb R^d))\to \mathbb R^{d-1}$ the linear operator given by
\[L_{k,l}^{(i)}(a)=\left(S_{k,l}^{(1;i)}(a),\cdots, \widehat{S_{k,l}^{(i;i)}(a)},\cdots,S_{k,l}^{(d;i)}(a)\right),\]
where $S_{k,l}^{(j;i)}(a)$ is the suggestive notation (as in Theorem \ref{thm:main}) that represents the component of $a$ on the $(k+l+1)$-tensor basis  $e_i^{\otimes k}\otimes e_j \otimes e_i^{\otimes l}$, i.e., $S_{k,l}^{(j;i)}(a) = \langle e_i^{\otimes k}\otimes e_j \otimes e_i^{\otimes l}, a \rangle$. One good property of this operator is that it stays linear when composing with the signature map.

\begin{prop}\label{prop:linear}
   Let $d \ge 2$. For each $k,l\in \mathbb N$ and $i\in \{1,\cdots, d\}$, the composite map 
    $$L_{k,l}^{(i)}\circ S \circ \iota^{(i)}: \mathcal P^1(\mathbb R^{d-1})\to \mathbb R^{d-1}$$
    is a linear operator.
\end{prop}
\begin{proof}
    Using the above notation, for any $\gamma = (x_1(\cdot),\cdots,\widehat{x_i(\cdot)},\cdots, x_d(\cdot)) \in \cP^1(\R^{d-1})$, we see that
    \[(L_{k,l}^{(i)}\circ S \circ \iota^{(i)})(\gamma)=\left(S_{k,l}^{(1;i)}(\iota^{(i)}(\gamma)),\cdots, \widehat{S_{k,l}^{(i;i)}(\iota^{(i)}(\gamma))},\cdots,S_{k,l}^{(d;i)}(\iota^{(i)}(\gamma))\right).\]
    Then from the definition, for each $j \neq i$, it holds that
    $$
    S_{k,l}^{(j;i)}(\iota^{(i)}(\gamma)) = \int_{0<t_1<\ldots<t_k<s<t_{k+1}<\ldots<t_{k+l+1}<1} dt_1\cdots dt_k x_j^\prime(s) dt_{k+1} \cdots dt_{k+l+1}.
    $$
    This indeed implies the linearity of the mapping $L_{k,l}^{(i)}\circ S \circ \iota^{(i)}$, for instance, for any real number $\lambda \in \R$, one has
    \begin{align*}
        S_{k,l}^{(j;i)}(\iota^{(i)}(\lambda \gamma)) &= \int dt_1\cdots dt_k (\lambda x_j^\prime)(s) dt_{k+1} \cdots dt_{k+l+1} \\
        &=\lambda \int dt_1\cdots dt_k x_j^\prime(s) dt_{k+1} \cdots dt_{k+l+1}\\
        &=\lambda S_{k,l}^{(j;i)}(\iota^{(i)}(\gamma)).
    \end{align*}
The linearity in the addition can be proved similarly.
\end{proof}

By pulling back the Euclidean norm on $\mathbb R^{d-1}$, we define the following family of semi-norms $||\cdot||_{k,l}^{(i)}$ on $T((\mathbb R^d))$ as
\[||a||_{k,l}^{(i)}:=|L_{k,l}^{(i)}(a)|,\;\forall a\in T((\mathbb R^d)).\]
By further renormalizing and taking certain limits of this family, we consider the following two asymptotic norms.

\begin{definition}\label{def: asymptotic norms}
    For each $i\in \{1,\cdots, d\}$ and $a\in T((\mathbb R^d))$, we define the (index-$i$) asymptotic supremum semi-norm as
    \[||a||_{AS}^{(i)}:=\varlimsup_{n\to +\infty}\left((n+1)!\cdot \sup_{\substack{k+l=n\\k,l\in \mathbb N}}||a||_{k,l}^{(i)}\right)\]
    and its asymptotic $L^1$ semi-norm as
    \[||a||_{AL^1}^{(i)}:=\varlimsup_{n\to +\infty}\left((n!)\cdot \sum_{k=1}^n||a||_{k,n-k}^{(i)}\right).\]   
\end{definition}

We now show that the above asymptotic semi-norms play an essential role for establishing the continuity of $S^{-1}$ for time-augmented curves. 
\begin{thm}\label{thm:norm} Let $d \ge 2$. The functions $||\cdot||_{AS}^{(i)}$ and $||\cdot||_{AL^1}^{(i)}$ satisfy the following properties:
    \begin{enumerate}
        \item They are semi-norms on $T((\mathbb R^d))$ which take values in $[0,+\infty]$.
        \item For any $\gamma_1,\gamma_2\in \cL^{(i)}_0(\mathbb R^d)$, and any $\lambda_1,\lambda_2\in \mathbb R$, we have
        \[||S(\lambda_1\gamma_1+\lambda_2\gamma_2)||_{AS}^{(i)}=||\lambda_1S(\gamma_1)+\lambda_2 S(\gamma_2)||_{AS}^{(i)},\textrm{ and }\]
        \[||S(\lambda_1\gamma_1+\lambda_2\gamma_2)||_{AL^1}^{(i)}=||\lambda_1S(\gamma_1)+\lambda_2 S(\gamma_2)||_{AL^1}^{(i)}.\]
        \item The semi-norms $||\cdot||_{AS}^{(i)}$ and $||\cdot||_{AL^1}^{(i)}$ induce metrics $d_{AS}$ and $d_{AL^1}$ on the image set $S \circ \iota^{(i)} (\cP^1(\R^{d-1}))$,  and
        the composite map $$S\circ \iota^{(i)}:(\mathcal P^1(\mathbb R^{d-1}),d_{C^1}) \to (T((\mathbb R^d)),d_{AS})$$
        and 
        $$S\circ \iota^{(i)}:(\mathcal P^1(\mathbb R^{d-1}),d_{BV}) \to (T((\mathbb R^d)),d_{AL^1})$$
        are isometries. 
        \item The semi-norms $||\cdot||_{AS}^{(i)}$ and $||\cdot||_{AL^1}^{(i)}$ induce metrics $d_{AS}$ and $d_{AL^1}$ on the image set $S (\cL^{(i)}_0(\R^{d}))$,  and
        the composite map $$S:(\cL^{(i)}_0(\R^{d}),d_{C^1}) \to (T((\mathbb R^d)),d_{AS})$$
        and 
        $$S:(\cL^{(i)}_0(\R^{d}),d_{BV}) \to (T((\mathbb R^d)),d_{AL^1})$$
        are isometries. 
    \end{enumerate}
\end{thm}

\begin{proof}
    For (1), we first observe that for each $k,l\in \mathbb N$, the function $||\cdot||_{k,l}^{(i)}$ is defined as a pull-back norm by the linear operator $L_{k,l}^{(i)}$, hence it must be a semi-norm. Then we have for any $a,b\in T((\mathbb R^d))$ that
     \begin{align*}
       ||a+b||_{AS}^{(i)}&= \varlimsup_{n\to +\infty}\left((n+1)!\cdot \sup_{\substack{k+l=n\\k,l\in \mathbb N}}||a+b||_{k,l}^{(i)}\right)\\ 
       &\leq \varlimsup_{n\to +\infty}\left((n+1)!\cdot \sup_{\substack{k+l=n\\k,l\in \mathbb N}}\left(||a||_{k,l}^{(i)}+||b||_{k,l}^{(i)}\right)\right)\\
       &\leq \varlimsup_{n\to +\infty}\left((n+1)!\cdot \sup_{\substack{k+l=n\\k,l\in \mathbb N}}||a||_{k,l}^{(i)}\right)+\varlimsup_{n\to +\infty}\left((n+1)!\cdot \sup_{\substack{k+l=n\\k,l\in \mathbb N}}||b||_{k,l}^{(i)}\right)\\
       &=||a||_{AS}^{(i)}+||b||_{AS}^{(i)},
    \end{align*}
    and similarly,
\begin{align*}
       ||a+b||_{AL^1}^{(i)}&= \varlimsup_{n\to +\infty}\left(n!\cdot \sum_{k=1}^n||a+b||_{k,n-k}^{(i)}\right)\\ 
       &\leq \varlimsup_{n\to +\infty}\left(n!\cdot \sum_{k=1}^n\left(||a||_{k,n-k}^{(i)}+||b||_{k,n-k}^{(i)}\right)\right)\\
       &\leq \varlimsup_{n\to +\infty}\left(n!\cdot \sum_{k=1}^n||a||_{k,n-k}^{(i)}\right)+\varlimsup_{n\to +\infty}\left(n!\cdot \sum_{k=1}^n||b||_{k,n-k}^{(i)}\right)\\
       &=||a||_{AL^1}^{(i)}+||b||_{AL^1}^{(i)}.
    \end{align*}
For any $a\in T((\mathbb R^d))$ and $\lambda\in \mathbb R$, we have that
\begin{align*}
    ||\lambda a||_{AS}^{(i)}&=\varlimsup_{n\to +\infty}\left((n+1)!\cdot \sup_{\substack{k+l=n\\k,l\in \mathbb N}}||\lambda a||_{k,l}^{(i)}\right)\\
    &=|\lambda|\cdot \varlimsup_{n\to +\infty}\left((n+1)!\cdot \sup_{\substack{k+l=n\\k,l\in \mathbb N}}||a||_{k,l}^{(i)}\right)\\
    &=|\lambda|\cdot ||a||_{AS}^{(i)},
\end{align*}
and similarly,
\begin{align*}
    ||\lambda a||_{AL^1}^{(i)}&=\varlimsup_{n\to +\infty}\left(n!\cdot \sum_{k=1}^n||\lambda a||_{k,n-k}^{(i)}\right)\\
    &=|\lambda|\cdot \varlimsup_{n\to +\infty}\left(n!\cdot \sum_{k=1}^n||a||_{k,n-k}^{(i)}\right)\\
    &=|\lambda|\cdot ||a||_{AL^1}^{(i)}.
\end{align*}
Thus, both $||\cdot||_{AS}^{(i)}$ and $||\cdot||_{AL^1}^{(i)}$ are semi-norms.

    For (2), it follows from Proposition \ref{prop:linear} that for each $k,l\in \mathbb N$, $\gamma_1,\gamma_2 \in \cL^{(i)}_0(\R^d) = \iota^{(i)}(\cP^1(\R^{d-1}))$ and any real numbers $\lambda_1,\lambda_2 \in \R$, we have
    \begin{align*}
    ||S(\lambda_1\gamma_1+\lambda_2\gamma_2)||_{k,l}^{(i)}&=|(L_{k,l}^{(i)}\circ S)(\lambda_1\gamma_1+\lambda_2\gamma_2)|\\
    &=|\lambda_1(L_{k,l}^{(i)}\circ S)(\gamma_1)+\lambda_2(L_{k,l}^{(i)}\circ S)(\gamma_2)|\\ &=|L_{k,l}^{(i)}\left(\lambda_1S(\gamma_1)+\lambda_2S(\gamma_2)\right)|\\
    &=||\lambda_1S(\gamma_1)+\lambda_2S(\gamma_2)||_{k,l}^{(i)}.
    \end{align*}
    This implies both
    \[||S(\lambda_1\gamma_1+\lambda_2\gamma_2)||_{AS}^{(i)}=||\lambda_1S(\gamma_1)+\lambda_2 S(\gamma_2)||_{AS}^{(i)},\;\textrm{ and }\]
   \[||S(\lambda_1\gamma_1+\lambda_2\gamma_2)||_{AL^1}^{(i)}=||\lambda_1S(\gamma_1)+\lambda_2 S(\gamma_2)||_{AL^1}^{(i)}.\]

    For (3), we choose an arbitrary $\widehat\gamma=(x_1(\cdot),\cdots,\widehat{x_i(\cdot)},\cdots, x_d(\cdot))$ in $\mathcal P^1(\mathbb R^{d-1})$, and denote $\gamma=\iota^{(i)}(\widehat \gamma)=(x_1(\cdot),\cdots,x_i(\cdot),\cdots,x_d(\cdot))$ where $x_i(s)=s$. Then we have for any $k,l\in \mathbb N$ that
    \begin{align}\label{eq:thm-norm-1}
    \begin{split}
         ||(S\circ\iota^{(i)})(\widehat{\gamma})||_{k,l}^{(i)}&=|(L_{k,l}^{(i)}\circ S\circ\iota^{(i)})(\widehat{\gamma})|\\
        &=\Bigg|\left(S_{k,l}^{(1;i)}(\gamma),\cdots, \widehat{S_{k,l}^{(i;i)}(\gamma)},\cdots,S_{k,l}^{(d;i)}(\gamma)\right)\Bigg|.
    \end{split}    
    \end{align}
    We first show that $||(S\circ\iota^{(i)})(\widehat{\gamma})||_{AS}^{(i)}=||\widehat \gamma'||_\infty$.

    For the ``$\leq$'' part, we may assume that the asymptotic norm is achieved by a subsequence $q_n\to +\infty$ such that
    \begin{equation}\label{eq:thm-norm-2}
        ||(S\circ\iota^{(i)})(\widehat{\gamma})||_{AS}^{(i)}=\lim_{n\to +\infty}\left((q_n+1)!\cdot ||(S\circ\iota^{(i)})(\widehat{\gamma})||_{k_n,l_n}^{(i)}\right)
    \end{equation}
    where $||(S\circ\iota^{(i)})(\widehat{\gamma})||_{k_n,l_n}^{(i)}$ achieves for each $q_n$ the following supremum
    $$\sup_{\substack{k+l=q_n\\k,l\in \mathbb N}}||(S\circ\iota^{(i)})(\widehat{\gamma})||_{k,l}^{(i)}.$$
    Since $\frac{k_n}{q_n}\in [0,1]$, by possibly passing onto another subsequence, which by abuse of notation we still assume to be the triple $(k_n,l_n,q_n)$, we have $\frac{k_n}{q_n}\to t_0$ for some $t_0\in [0,1]$. By Theorem \ref{thm:main}, we have that
    \[(q_n+1)!\cdot S_{k_n,l_n}^{(j;i)}(\gamma)\to x_j'(t_0) \textrm{ as } n\to +\infty\]
    holds for all $j\neq i$. Thus, using \eqref{eq:thm-norm-1} and \eqref{eq:thm-norm-2}, we have
    \begin{align*}
        ||(S\circ\iota^{(i)})(\widehat{\gamma})||_{AS}^{(i)}&=\lim_{n\to \infty}\left((q_n+1)!\cdot ||(S\circ\iota^{(i)})(\widehat{\gamma})||_{k_n,l_n}^{(i)}\right)\\
        &=\lim_{n\to \infty}\left((q_n+1)!\cdot\Bigg|\left(S_{k_n,l_n}^{(1;i)}(\gamma),\cdots, \widehat{S_{k_n,l_n}^{(i;i)}(\gamma)},\cdots,S_{k_n,l_n}^{(d;i)}(\gamma)\right)\Bigg|\right)\\
&=\Big|\left(x_1'(t_0),\cdots,\widehat{x_i'(t_0)},\cdots,x'_d(t_0)\right)\Big|\\
&\leq ||\widehat \gamma'||_\infty.
    \end{align*}
    
    For the ``$\geq$'' part, we assume that $||\widehat \gamma'||_\infty$ is achieved at $s=s_0\in [0,1]$ so that
    \[||\widehat \gamma'||_\infty=\Big|\left(x'_1(s_0),\cdots,\widehat{x'_i(s_0)},\cdots, x'_d(s_0)\right)\Big|.\]
    Then we take $k_n=\floor{ns_0}$ and $l_n=n-k_n$ so that $(k_n,n)$ is a rational approximation of $s_0$. Apply Theorem \ref{thm:main}, we have that for all $j\neq i$,
    \[x_j'(s_0)=\lim_{n\to +\infty}\left((n+1)!\cdot S_{k_n,l_n}^{(j;i)}(\gamma)\right).\]
    It follows that
    \begin{align*}
        ||\widehat \gamma'||_\infty&=\Big|\left(x'_1(s_0),\cdots,\widehat{x'_i(s_0)},\cdots, x'_d(s_0)\right)\Big|\\
        &=\lim_{n\to \infty}\left((n+1)!\cdot\Bigg|\left(S_{k_n,l_n}^{(1;i)}(\gamma),\cdots, \widehat{S_{k_n,l_n}^{(i;i)}(\gamma)},\cdots,S_{k_n,l_n}^{(d;i)}(\gamma)\right)\Bigg|\right)\\
        &=\lim_{n\to \infty}\left((n+1)!\cdot||(S\circ\iota^{(i)})(\widehat{\gamma})||_{k_n,l_n}^{(i)}\right)\\
        &\leq \lim_{n\to \infty}\left((n+1)!\cdot\sup_{\substack{k,l\in \mathbb N\\ k+l=n}}||(S\circ\iota^{(i)})(\widehat{\gamma})||_{k,l}^{(i)}\right)\\
        &\leq ||(S\circ\iota^{(i)})(\widehat{\gamma})||_{AS}^{(i)},
    \end{align*}
    where the third equality uses \eqref{eq:thm-norm-1} and the last inequality uses the definition of the asymptotic norm.

    Next, We show that $||(S\circ\iota^{(i)})(\widehat{\gamma})||_{AL^1}^{(i)}=||\widehat \gamma'||_{L^1}$. 

    First, we can approximate $||\widehat \gamma'||_{L^1}$ with the Riemann sum, and for any $\epsilon>0$, there exists $N>0$ such that when $n>N$, we have
    \begin{equation}\label{eq:thm-norm-riemann}
        \Bigg|\sum_{k=1}^n \bigg|\widehat{\gamma}'(\frac{k}{n})\bigg|\cdot \frac{1}{n}-||\widehat \gamma'||_{L^1}\Bigg|<\epsilon.
    \end{equation}
    Secondly, we assume that the asymptotic norm is achieved by a subsequence $q_n\to +\infty$ such that
    \begin{equation}\label{eq:thm-norm-3}
        ||(S\circ\iota^{(i)})(\widehat{\gamma})||_{AL^1}^{(i)}=\lim_{n\to +\infty}\left(q_n!\cdot \sum_{k=1}^{q_n} ||(S\circ\iota^{(i)})(\widehat{\gamma})||_{k,q_n-k}^{(i)}\right)
    \end{equation}
    We now define the sequence of maps $(p_n,q_n):[0,1]\to \mathbb N\times \mathbb N^*$ as $q_n(s)=q_n$ and $p_n(s)=\floor{sq_n}$ similar to Example \ref{ex:natural-r-approx}. Then $(p_n,q_n)$ is a uniform rational approximation, thus applying Theorem \ref{thm:main}, we have for any $j\neq i$
    \[\sup_{s\in [0,1]}\bigg|x_j'(s)-(q_n+1)!\cdot S_{p_n(s),q_n-p_n(s)}^{(j;i)}(\gamma)\bigg|<\frac{\epsilon}{\sqrt d}.\]
    Implementing at the sample points $s=k/q_n$ for each $k=1,\cdots, q_n$, we have
    \begin{equation}\label{eq:thm-norm-4}
        \bigg|x_j'(k/q_n)-(q_n+1)!\cdot S_{k,q_n-k}^{(j;i)}(\gamma)\bigg|<\frac{\epsilon}{\sqrt d}
    \end{equation}
    holds for all $k=1,\cdots, q_n$ and $j\neq i$. This implies that for any $k=1,\cdots, q_n$, we have
    \begin{align*}
        &\quad\;\Bigg|\big|\widehat{\gamma}'(k/q_n)\big|-(q_n+1)!\cdot ||(S\circ\iota^{(i)})(\widehat{\gamma})||_{k,q_n-k}^{(i)}\Bigg|\\
        &=\Bigg|\big|\widehat{\gamma}'(k/q_n)\big|-(q_n+1)!\cdot \Big|\left(S_{k,q_n-k}^{(1;i)}(\gamma),\cdots, \widehat{S_{k,q_n-k}^{(i;i)}(\gamma)},\cdots,S_{k,q_n-k}^{(d;i)}(\gamma)\right)\Big|\Bigg|\\
        &\leq \Bigg|\widehat{\gamma}'(k/q_n)-(q_n+1)!\cdot \left(S_{k,q_n-k}^{(1;i)}(\gamma),\cdots, \widehat{S_{k,q_n-k}^{(i;i)}(\gamma)},\cdots,S_{k,q_n-k}^{(d;i)}(\gamma)\right)\Bigg|\\
        &\leq \sqrt{d}\cdot\max_{j\neq i} \bigg|x_j'(k/q_n)-(q_n+1)!\cdot S_{k,q_n-k}^{(j;i)}(\gamma)\bigg|\\
        &< \epsilon,
    \end{align*}
where the first equality uses \eqref{eq:thm-norm-1}, the second inequality uses the triangle inequality, and the last inequality uses \eqref{eq:thm-norm-4}. It follows that
\begin{align*}
    &\quad\;\Bigg|\sum_{k=1}^{q_n}\big|\widehat{\gamma}'(k/q_n)\big|\cdot \frac{1}{q_n}-(q_n+1)!\cdot \left(\sum_{k=1}^{q_n}||(S\circ\iota^{(i)})(\widehat{\gamma})||_{k,q_n-k}^{(i)}\right)\cdot \frac{1}{q_n} \Bigg|\\
    &\leq \sum_{k=1}^{q_n}\Bigg|\big|\widehat{\gamma}'(k/q_n)\big|-(q_n+1)!\cdot ||(S\circ\iota^{(i)})(\widehat{\gamma})||_{k,q_n-k}^{(i)}\Bigg|\cdot \frac{1}{q_n}\\
    &<\epsilon.
\end{align*}
Combining this with \eqref{eq:thm-norm-riemann} and \eqref{eq:thm-norm-3}, and let $q_n\to +\infty$, we conclude that
\[||(S\circ\iota^{(i)})(\widehat{\gamma})||_{AL^1}^{(i)}=||\widehat \gamma'||_{L^1}.\]
From the above we can easily see that if an element $a \in S \circ \iota^{(i)}(\cP^1(\R^{d-1})$ with $a = S \circ \iota^{(i)}(\hat \gamma)$ for some curve $\hat \gamma \in \cP^1(\R^{d-1})$ satisfies that $\|a\|_{AS}^{(i)} = 0$, then it must hold that $\|\hat \gamma^\prime\|_{\infty} = 0$, which implies that $\hat \gamma$ must be a constant curve and consequently $a = (1,0,0,\ldots,) \in T((\R^d))$. Similarly, if $\|a\|^{(i)}_{AL^1}=0$, then it must hold that $\|\hat \gamma^\prime\|_{L^1} = L(\hat \gamma) = 0$, that is, $\hat \gamma$ is a constant curve as well and therefore $a = (1,0,0,\ldots,) \in T((\R^d))$. Hence $\|\cdot\|_{AS}^{(i)}$ and $\|\cdot\|_{AL^1}^{(i)}$ are non-degenerate on the image set $S \circ \iota^{(i)}(\cP^1(\R^{d-1})$, so that they induce metrics $d_{AS}(a,b) = \|a - b\|^{(i)}_{AS}$ and $d_{AL^1}(a,b) = \|a - b\|^{(i)}_{AL^1}$
on $S \circ \iota^{(i)}(\cP^1(\R^{d-1})$.

Finally, in view of (2), for $\hat \gamma_1, \hat \gamma_2 \in \cP^1(\R^{d-1})$, we have (recall that $\cL^{(i)}_0(\R^d) = \iota^{(i)}(\cP^1(\R^{d-1}))$)
$$
\|S\circ \iota^{(i)}(\hat \gamma_1 -\hat \gamma_2)\|_{AS}^{(i)} =\|S\circ \iota^{(i)}(\hat \gamma_1) - S \circ \iota^{(i)}(\hat \gamma_2)\|_{AS}^{(i)},
$$
which certainly implies that
\begin{align*}
  \|S\circ \iota^{(i)}(\hat \gamma_1) - S \circ \iota^{(i)}(\hat \gamma_2)\|_{AS}^{(i)} &=   \|S\circ \iota^{(i)}(\hat \gamma_1 -\hat \gamma_2)\|_{AS}^{(i)}\\
  &=\|\hat \gamma_1^\prime - \hat \gamma_2^\prime\|_\infty\\
  &=d_{C^1}(\hat \gamma_1, \hat \gamma_2).
\end{align*}
The same argument also provides that $\|S\circ \iota^{(i)}(\hat \gamma_1) - S \circ \iota^{(i)}(\hat \gamma_2)\|_{AL^1}^{(i)} = \|\hat \gamma_1^\prime - \hat \gamma^\prime_2\|_{L^1} = d_{BV}(\hat \gamma_1, \hat \gamma_2)$. This completes the proof of (3).

Now (4) follows immediately from (3) by noting that $\cL^{(i)}_0(\R^d) = S \circ \iota^{(i)}(\cP^1(\R^{d-1}))$.
\end{proof}

\begin{cor}
 Let $d \ge 2$.  
 Let $\tilde S := S \circ \iota^{(1)}: \cP^1(\R^{d-1}) \to T((\R^d))$ be the ``canonical time-augmented signature'' on $\cP^1(\R^{d-1})$, then 
     $$
     \tilde S: (\cP^1(\R^{d-1}),d_{C^1}) \to (\tilde S(\cP^1(\R^{d-1})),d_{AS})
     $$
     and 
     $$
     \tilde S: (\cP^1(\R^{d-1}),d_{BV}) \to (\tilde S(\cP^1(\R^{d-1})),d_{AL^1})
     $$
     are homeomorphism. In particular, $$\tilde S^{-1}:  (\tilde S(\cP^1(\R^{d-1})),d_{AS}) \to (\cP^1(\R^{d-1}),d_{C^1})$$ and $\tilde S^{-1}:  (\tilde S(\cP^1(\R^{d-1})),d_{AL^1}) \to (\cP^1(\R^{d-1}),d_{BV})$ are continuous.
\end{cor}

\section{Effective bounds and Diophantine approximation}\label{sec:effective}

In this section we will express the convergence rate of our signature inversion scheme $\cI$ and the modulus of continuity of the inverse $S^{-1}$ for $x_i$-linear curves in terms of the modulus of continuity of their derivatives.

First, we present an effective convergence bound in Theorem \ref{thm:main} for $x_i$-linear curves with H\"older continuous derivatives.

For each $\alpha\in (0,1]$, we denote $\mathcal L^{(i)}_\alpha= \mathcal L^{(i)}\cap C^{1,\alpha}$, that is,
\[\mathcal L^{(i)}_\alpha=\{\gamma\in \mathcal L^{(i)}\;|\; \gamma' \textrm{ is }\alpha\textrm{-H\"older}\}.\]
As before, for each $\gamma \in \cP^1(\R^d)$, we write $\gamma(s)=\left(x_1(s),\cdots, x_d(s)\right)$. Note that $\gamma'$ being $\alpha$-H\"older implies that each $x_i$-component is $\alpha$-H\"older. We use
$$
\|f\|_{\alpha} := \inf \bigg\{C \ge 0: \sup_{s \neq t \in [0,1]} \frac{|f(t) - f(s)|}{|t-s|^\alpha} \le C  \bigg\}
$$
for the H\"older norm of function $f$.
We have the following quantitative convergence.

\begin{thm}\label{thm:quantitative}
Let $\alpha\in (0,1]$ and $\gamma\in \mathcal L^{(i)}_\alpha(\R^d)$. Then for any $\epsilon_0\in (0,\frac{1}{2})$, any $x\in[0,1]$, any $j\neq i$, and any rational approximation $\frac{p_n}{q_n}\in [0,1]$ of $x$ with $p_n,q_n\in \mathbb N$ that satisfies
\begin{equation}\label{eq:rational approx condition}
   \Bigg|\frac{p_n}{q_n}- x\Bigg|<q_n^{-\frac{1}{2}}\textrm{ and } q_n\to \infty\; \textrm{as }n\rightarrow \infty, 
\end{equation}
we have
\[\Bigg|x_j'(x)-\frac{(q_n+1)!\cdot S^{(j;i)}_{p_n,q_n-p_n}}{(S^{(1)}_i)^{q_n}}\Bigg|\leq C_1\cdot q_n^{(-\frac{1}2+\epsilon_0)\alpha},\]
where $C_1$ is a constant that only depends on $\epsilon_0$, $\alpha$, $||\gamma'||_\infty$, and $\|\gamma^\prime\|_\alpha$.
\end{thm}

The rational approximation stated above always exists, and it is closely related to the Diophantine approximation in number theory.
For any rational point $x=\frac{p}{q}$ where $p,q\in \mathbb N$, we can simply take the integer pair $(p_n,q_n)=(pn,qn)$. For any irrational $x\in [0,1]$, for example, one can take the naive approximation $p_n=\floor{nx}, q_n=n$, and the resulting pairs $(p_n,q_n)$ satisfy
    \[\Bigg|\frac{p_n}{q_n}-x\Bigg|< \frac{1}{q_n},\]
which certainly satisfies the condition of Theorem \ref{thm:quantitative}. Using the decimal approximation, one can obtain a similar bound as above. Moreover, according to the Dirichlet's approximation theorem \cite{Dirichlet1842}, there exists a sequence of integer pairs $(p_n,q_n)$ such that
    \[\Bigg|\frac{p_n}{q_n}-x\Bigg|< \frac{1}{q_n^2}.\]
The best Diophantine approximation can be obtained using the continued fraction of $x$. For any real number $x\in [0,1]$, it can be uniquely expressed by the (simple) continued fraction $(a_0;a_1,a_2,\cdots,a_n,\cdots)$ with $a_0 \in \N$ and $a_i\in \N^*$ for $i\geq 1$, that is, we have
\[x=a_0+\frac{1}{a_1+\frac{1}{a_2+\cdots+\frac{1}{a_n+\cdots}}}.\]
If we denote the $n$-th truncation by
\[\frac{p_n}{q_n}=a_0+\frac{1}{a_1+\frac{1}{a_2+\cdots+\frac{1}{a_n}}},\]
then it satisfies the Dirichlet principle that
 \[\Bigg|\frac{p_n}{q_n}-x\Bigg|< \frac{1}{q_n^2}.\]
See for example \cite[Section 3.1]{Einsiedler-Ward:ErgodicTheory}. 
We summarize that the rational approximation condition \eqref{eq:rational approx condition} in Theorem \ref{thm:quantitative} is a very loose condition, and all the methods given above (and many others) will satisfy it.

The main idea of the proof of Theorem \ref{thm:quantitative} lies in the following quantitative improvement of Theorem \ref{thm:convergence}.

\begin{thm}\label{thm:quantitative-convergence}
Let $\alpha\in (0,1]$. For any $\alpha$-H\"older continuous function $f$ on $[0,1]$, any $\epsilon_0\in (0,\frac{1}2)$, any $x\in[0,1]$, and any rational approximation $\frac{p_n}{q_n}\in [0,1]$ of $x$ with $p_n,q_n\in \mathbb N$ that satisfies
\[\Bigg|\frac{p_n}{q_n}- x\Bigg|<q_n^{-\frac{1}{2}}\textrm{ and } q_n\to +\infty\; \textrm{as }n\rightarrow \infty, \]
we have
    \[\Bigg|\int_{0}^1\rho_{p_n,q_n-p_n}(s)f(s)ds-f(x)\Bigg|\leq C_1\cdot q_n^{(-\frac{1}2+\epsilon_0)\alpha},\]
where $C_1$ is a constant that only depends on $\epsilon_0$, $\alpha$, $||f||_\infty$, and $\|f\|_\alpha$.
\end{thm}
\begin{proof}
    The proof is a refined estimate of that in Theorem \ref{thm:convergence}. We set $M=||f||_\infty+1>0$. 

    By applying Proposition \ref{prop:key} for a fixed $\epsilon_0 \in (0,\frac{1}{2})$, $k_n=p_n$, $l_n=q_n-p_n$, we see that there exists $q_0 \in \N^*$ such that if $q_n>q_0$, and $|s-\frac{p_n}{q_n}|\geq q_n^{-\frac{1}2+\epsilon_0}$, then we have
    \begin{equation}\label{eq:thm-quant-conv-1}
        \rho_{p_n,q_n-p_n}(s)\leq 3q_n^{3/2}\exp\left(-\frac{1}{18}q_n^{2\epsilon_0}\right).
    \end{equation}
    Since $q_n\to +\infty$ and $|\frac{p_n}{q_n}- x|<q_n^{-\frac{1}{2}}$ as $n \to \infty$ by our hypothesis, there exists $n_0 \in \N^*$ such that when $n>n_0$, we have $q_n>q_0$ and $|\frac{p_n}{q_n}-x|<q_n^{-\frac{1}{2}+\epsilon_0}$. On the other hand, we note that the bound
    \[3q_n^{3/2}\exp\left(-\frac{1}{18}q_n^{2\epsilon_0}\right)\lesssim q_n^{-\beta}\]
    holds for any $\beta>0$, and in particular if we choose $\beta=\alpha(\frac{1}2-\epsilon_0)$, which together with \eqref{eq:thm-quant-conv-1} gives that there is some $n_1 \in \N^*$ such that for all $n > n_1$, we have
    \begin{equation}\label{eq:thm-quant-conv-2}
        \rho_{p_n,q_n-p_n}(s)\leq q_n^{-\alpha(\frac{1}{2}-\epsilon_0)},
    \end{equation}
    as long as $|s - \frac{p_n}{q_n}| \ge q_n^{-\frac{1}{2}+\epsilon_0}$.
    Furthermore, if we set $\delta_n:=2q_n^{-\frac{1}{2}+\epsilon_0}$, we have by triangle inequality that for all $n > n_1$ and for all $s \in [0,1]$ with $|s-x|\ge \delta_n$,
   \begin{equation}\label{eq:thm-quant-conv-3}
       \Bigg|s-\frac{p_n}{q_n}\Bigg|\geq \Bigg|s-x\Bigg|-\Bigg|x-\frac{p_n}{q_n}\Bigg|\geq \delta_n - \frac{\delta_n}{2}\geq q_n^{-\frac{1}{2}+\epsilon_0},
   \end{equation} 
   so that the bound \eqref{eq:thm-quant-conv-2} holds for all such $s$.
    
     Now we can give the following estimates for any $n>n_1$ that
    \begin{align*}
        \Bigg|\int_{0}^1\rho_{k_n,l_n}(s)f(s)ds-f(x)\Bigg|&\leq \int_{0}^1\rho_{k_n,l_n}(s)\Big|f(s)-f(x)\Big|ds\\
        &=\int_{|s-x|<\delta_n}\rho_{k_n,l_n}(s)\Big|f(s)-f(x)\Big|ds\\
        &\quad\quad+\int_{|s-x|\geq\delta_n}\rho_{k_n,l_n}(s)\Big|f(s)-f(x)\Big|ds\\
        &<\|f\|_\alpha \delta_n^\alpha \int_{|s-x|<\delta_n}\rho_{k_n,l_n}(s)ds \\
        &\quad \quad +2M \int_{|s-x|\geq \delta_n}\rho_{k_n,l_n}(s)ds\\
        &<\|f\|_\alpha \delta_n^\alpha + 2Mq_n^{-\alpha(\frac{1}{2}-\epsilon_0)}\\
        &\leq C_1\cdot q_n^{(-\frac{1}2+\epsilon_0)\alpha},
    \end{align*}
    where we used the H\"older continuity of $f$ to get the second inequality, the bound \eqref{eq:thm-quant-conv-2} in the third inequality and the assumption $\delta_n = 2q_n^{-\frac{1}{2}+\epsilon_0}$ in the last one. Clearly $C_1$ only depends on $\epsilon_0$, $\alpha$, $||f||_\infty$, and $\|f\|_\alpha$.
\end{proof}

Now we present the proof of Theorem \ref{thm:quantitative}.
\begin{proof}
    Similar to the proof of Theorem \ref{thm:main}, we may assume without loss of generality that $\gamma$ is $x_1$-increasing, and moreover by projecting the curve to the $(x_1,x_j)$-coordinate, we may assume the projected curve is of the form
    \[\gamma'(s)=\left(C_0s, y(s)\right),\]
    where $C_0=S^{(1)}_1$ and $y(s)=x_j'(s)$ which by assumption is $\alpha$-H\"older. Apply Theorem \ref{thm:quantitative-convergence} for the $\alpha$-H\"older function $y(s)$, we obtain that for any $\epsilon_0\in (0,\frac{1}2)$, any $x\in[0,1]$, and any rational approximation $\frac{p_n}{q_n}\in [0,1]$ of $x$ with $p_n,q_n\in \mathbb N$ that satisfies
\[\Bigg|\frac{p_n}{q_n}- x\Bigg|<q_n^{-\frac{1}{2}}\textrm{ and } q_n\to +\infty\; \textrm{as }n\rightarrow +\infty, \]
we have
    \[\Bigg|\int_{0}^1\rho_{p_n,q_n-p_n}(s)y(s)ds-y(x)\Bigg|\leq C_1\cdot q_n^{(-\frac{1}2+\epsilon_0)\alpha},\]
    where $C_1$ is a constant that depends on $\epsilon_0$, $\alpha$, $||\gamma'||_\infty$, and the multiplicative constant of $\gamma'$ in the $\alpha$-H\"older condition.
    Finally, using Proposition \ref{prop:S_kl} and Definition \ref{def:rho}, the result follows.
\end{proof}

Finally, we give an effective convergence bound to Corollary \ref{cor:S-inv-uniform}, which gives a quantitative description of the signature inverse $S^{-1}$ on the image of some equicontinuous sets in $\cL^{(i)}_\alpha(\R^d)$ with respect to the projective norm $\|\cdot\|_{proj}$ on the target space and the $d_{C^1}$-metric on the domain. 
 Here we define for each $K>0$, the following subset
\[\mathcal L^{(i)}_{\alpha,K}(\R^d)=\{\gamma\in \mathcal L^{(i)}:\max\{\|\gamma^\prime\|_\infty, \|\gamma^\prime\|_\alpha\} \le K\}.\]
By the Arzela-Ascoli theorem, for any $K>0$, the set $\mathcal L^{(i)}_{\alpha,K}(\R^d)$ is relatively compact in $\mathcal L^{(i)}_{\alpha}(\R^d)$. Hence, when restricted to this set $\mathcal L^{(i)}_{\alpha,K}(\R^d)$, the signature map $S$ is a homeomorphism onto its image (See Remark \ref{remark:homeo}). In particular, $S^{-1}$ is continuous on this image, when we use the projective norm on the image. The following theorem gives a quantitative estimate on the modulus of continuity of $S^{-1}$.

\begin{thm}\label{thm:mod-cont-S-inverse}
    Let $\alpha\in (0,1]$, $K>0$, $\epsilon_0\in (0,\frac{1}2)$, and $\gamma\in \mathcal L^{(i)}_{\alpha,K}$.
    Denote $G^{(i)}_{\alpha,K} = S(\mathcal L^{(i)}_{\alpha,K}(\R^d))$ the image of $\mathcal L^{(i)}_{\alpha,K}(\R^d)$ under the signature map $S$. Then for any $\epsilon\in (0,1)$, there exists $\delta(\epsilon)>0$, which can be explicitly chosen as
    \[\delta<\frac{\overline C_2}{\sqrt d}\frac{(S^{(1)}_i(\gamma))^n}{(n+1)!} \epsilon\;\textrm{ and }\;n>\overline C_1\cdot \epsilon^{\frac{1}{\left(-\frac{1}2+\epsilon_0\right)\alpha}},\]
where $\overline C_1$ depends on $\epsilon_0$, $\alpha$, and $K$, and $\overline C_2$ depends on $K$ and $S^{(1)}_i(\gamma)$, such that 
    \[d_{C^1}(S^{-1}(\overline S), S^{-1}(S(\gamma))<\epsilon\]
    whenever $\overline S\in G^{(i)}_{\alpha,K}$ satisfies
    $||\overline S-S(\gamma)||_{proj}<\delta$.
\end{thm}

\begin{proof}
    For a given $\overline S \in G^{(i)}_{\alpha,K}$, we write it as $\overline S = S(\overline \gamma)$ for some $\overline \gamma \in \cL^{(i)}_{\alpha,K}(\R^d)$. 
    We write $\gamma=(x_1(\cdot),\cdots, x_d(\cdot))$ and $\overline \gamma=(\overline x_1(\cdot),\cdots,\overline x_d(\cdot))$. We write $\overline S^{(j;i)}_{k,l} = \langle e_i^{\otimes k} \otimes e_j \otimes e_i^{\otimes l}, \overline S \rangle$ and $\overline S^{(1)}_i = \langle e_i, \overline S\rangle$.
    
    For any $x\in [0,1]$, for simplicity we choose the naive approximation $p_n=\floor{nx}, q_n=n$, then it satisfies the condition of Theorem \ref{thm:quantitative-convergence}, thus applying the theorem for both $x_j'(x)$ and $\overline{x}_j'(x)$, we have for each $j\neq i$ that
    \begin{equation*}
        \Bigg|x_j'(x)-\frac{(n+1)!\cdot S^{(j;i)}_{p_n,n-p_n}}{(S^{(1)}_i)^{n}}\Bigg|\leq C_1\cdot n^{(-\frac{1}2+\epsilon_0)\alpha},
    \end{equation*}
    and
    \begin{equation*}
        \Bigg|\overline x_j'(x)-\frac{(n+1)!\cdot \overline S^{(j;i)}_{p_n,n-p_n}}{(\overline S^{(1)}_i)^{n}}\Bigg|\leq C_1\cdot n^{(-\frac{1}2+\epsilon_0)\alpha}.
    \end{equation*}
     Note that since both $\gamma,\overline \gamma$ belong to $\mathcal L^{(i)}_{\alpha,K}(\R^d)$, the above two inequalities share the same constant $C_1$, which depends only on $\epsilon_0$, $\alpha$, and $K$. Now for any $\epsilon\in (0,1)$, by choosing
    \begin{equation}\label{eq:thm-modcts-n}
        n>\left(\frac{\epsilon}{3\sqrt d\cdot C_1}\right)^{\frac{1}{\left(-\frac{1}2+\epsilon_0\right)\alpha}},
    \end{equation}
    we have both
    \begin{equation}\label{eq:thm-modcts-1}
        \Bigg|x_j'(x)-\frac{(n+1)!\cdot S^{(j;i)}_{p_n,n-p_n}}{(S^{(1)}_i)^{n}}\Bigg|\leq \frac{\epsilon}{3\sqrt d},
    \end{equation}
    and
    \begin{equation}\label{eq:thm-modcts-2}
        \Bigg|\overline x_j'(x)-\frac{(n+1)!\cdot \overline S^{(j;i)}_{p_n,n-p_n}}{(\overline S^{(1)}_i)^{n}}\Bigg|\leq \frac{\epsilon}{3\sqrt d}.
    \end{equation}
    Next, we estimate (for simplicity, we write $S_i$ for $S^{(1)}_i$ and $\overline S_i$ for $\overline S^{(1)}_i$ below)
    \begin{align}\label{eq:thm-modcts-3}
    \begin{split}
        &\quad \;\Bigg|\frac{(n+1)!\cdot S^{(j;i)}_{p_n,n-p_n}}{S_i^{n}}-\frac{(n+1)!\cdot \overline S^{(j;i)}_{p_n,n-p_n}}{\overline S_i^{n}}\Bigg|\\
        &=(n+1)!\cdot \Bigg|\frac{S^{(j;i)}_{p_n,n-p_n}\overline S_i^{n}-\overline S^{(j;i)}_{p_n,n-p_n}\overline S_i^{n}+\overline S^{(j;i)}_{p_n,n-p_n}\overline S_i^{n}-\overline S^{(j;i)}_{p_n,n-p_n}S_i^{n}}{S_i^n\cdot \overline S_i^n}\Bigg|\\
        &\leq (n+1)!\cdot \left(\frac{1}{ S_i^n}|\overline S^{(j;i)}_{p_n,n-p_n}-S^{(j;i)}_{p_n,n-p_n}|+\frac{|\overline S^{(j;i)}_{p_n,n-p_n}|\cdot |\overline S_i^{n-1}+\overline S_i^{n-2}S_i+\cdots+S_i^{n-1}|}{S_i^n\overline S_i^n}|\overline S_i-S_i|\right)\\
        &\leq \left(\frac{(n+1)!}{S_i^n}+\frac{K^{n+1}}{S_i^n\overline S_i^n}\cdot nK^{n}\right)\delta,
        \end{split}
    \end{align}
    where the last inequality uses Proposition \ref{prop:tensor-ineq} together with the following apriori estimates
    \[ \big|\overline S_{k,n-k}^{(j;i)}\big|\leq \frac{K^{n+1}}{(n+1)!},\quad S_i\leq K, \quad \overline S_i\leq K.\]
    If we choose $\delta$ small enough, say $\delta<\frac{S_i}2$, then it implies that
    $$\overline S_i\geq S_i-\big|\overline S_i- S_i\big|\geq S_i-\delta \geq \frac{S_i}2.$$
    Then for each fixed $K$ and $S_i$, as $n\to \infty$, the term $(n+1)!$ will dominate the growth as in \eqref{eq:thm-modcts-3}, so it further simplifies to
    \begin{align*}
        &\quad \;\Bigg|\frac{(n+1)!\cdot S^{(j;i)}_{p_n,n-p_n}}{S_i^{n}}-\frac{(n+1)!\cdot \overline S^{(j;i)}_{p_n,n-p_n}}{\overline S_i^{n}}\Bigg|\\
        &\leq \left(\frac{(n+1)!}{S_i^n}+\frac{K^{n+1}}{S_i^n\overline S_i^n}\cdot nK^{n}\right)\delta\\
        &<C_2\cdot \frac{(n+1)!}{S_i^n}\cdot \delta,        \end{align*}
    where $C_2$ is a constant depending on $K$ and $S_i$. Hence, by choosing $$\delta< \frac{1}{3C_2\sqrt d}\frac{S_i^n}{(n+1)!}\cdot \epsilon$$ (also choosing $C_2$ properly will guarantee $\delta<\frac{S_i}{2}$), we have
\begin{equation}\label{eq:thm-modcts-4}
    \Bigg|\frac{(n+1)!\cdot S^{(j;i)}_{p_n,n-p_n}}{S_i^{n}}-\frac{(n+1)!\cdot \overline S^{(j;i)}_{p_n,n-p_n}}{\overline S_i^{n}}\Bigg|<\frac{\epsilon}{3\sqrt d},
\end{equation}
as long as $\|\overline S - S(\gamma)\|_{proj} \le \delta$. 
Putting together with \eqref{eq:thm-modcts-1} and \eqref{eq:thm-modcts-2}, we have by the triangle inequality that
\[|x_j'(x)-\overline x_j'(x)|< \frac{\epsilon}{\sqrt d}, \quad \forall j\neq i,\;\forall x\in [0,1],\]
whenever $\|\overline S - S(\gamma)\|_{proj} \le \delta$ and $\delta$ is chosen to be
\[\delta<\frac{\overline C_2}{\sqrt d}\frac{S_i^n}{(n+1)!} \epsilon,\]
and 
\[n>\overline C_1\cdot \epsilon^{\frac{1}{\left(-\frac{1}2+\epsilon_0\right)\alpha}},\]
where $\overline C_1$ is a constant depending on $\epsilon_0$, $\alpha$, and $K$, and $\overline C_2$ depends on $K$ and $S_i$. We can also choose $\overline C_2$ properly so that $\delta<\frac{\epsilon}{\sqrt d}$ holds automatically.

Finally, since 
$$\big|\overline x_i'(x)-x_i'(x)\big|=\big|\overline S_i-S_i\big|\leq ||S(\overline \gamma)-S(\gamma)||_{proj}\le \delta<\frac{\epsilon}{\sqrt d},$$
we obtain
\begin{align*}
    \sup_{x\in [0,1]}|\overline \gamma'(x)-\gamma'(x)|&=\sup_{x\in [0,1]}\left(\sum_{j=1}^d \big|\overline x_j'(x)-x_j'(x)\big|^2\right)^{\frac{1}2}\\
    &\leq \left(\sup_{x\in [0,1]}\sup_{j\in \{1,\cdots, d\}}\big|\overline x_j'(x)-x_j'(x)\big|\right)\cdot \sqrt d\\
    &<\epsilon,
\end{align*}
which is exactly $d_{C^1}(S^{-1}(\overline S), S^{-1}(S(\gamma)) < \epsilon$, as claimed.
\end{proof}

\begin{remark}
 The above result may suggest that the signature inverse $S^{-1}$ may have a ``bad'' modulus of continuity when the image set is endowed with the projective norm, although the signature mapping $S$ is locally Lipschitz for $\tau_{proj}$, see Theorem \ref{thm:cont-of-S}.
\end{remark}

\bibliographystyle{alpha}
\bibliography{myref}
\vskip 0.5cm
\footnotesize{
CL: ShanghaiTech University, Pudong, Shanghai, China. E-mail: \verb|liuchong@shanghaitech.edu.cn|}
{SW: ShanghaiTech University, Pudong, Shanghai, China. E-mail: \verb|wangshi@shanghaitech.edu.cn|}
\end{document}